\documentclass[a4paper,reqno]{amsart}

\pdfoutput=1

\usepackage[unicode=true,hidelinks]{hyperref}
\usepackage{amssymb}
\usepackage{caption}

\usepackage[round]{natbib}

\title[Univalent Higher Categories via Complete Semi-Segal Types]{Univalent Higher Categories via Complete Semi-Segal Types}

\author{Paolo Capriotti \and Nicolai Kraus}

\thanks{%
This is the full version of the paper, containing all proofs.
An abridged version titled \emph{Univalent Higher Categories via Complete Semi-Segal Types} will be presented at the 45th ACM SIGPLAN Symposium on Principles of Programming Languages (POPL 2018). \\
Paolo Capriotti is supported by USAF, Airforce office for scientific research, award FA9550-16-1-0029, and Nicolai Kraus by the Engineering and Physical Sciences Research Council (EPSRC), grant reference EP/M016994/1.}

\usepackage{mathtools}
\usepackage{tikz}
\usetikzlibrary{arrows}
\usetikzlibrary{positioning}
\usetikzlibrary{calc}
\usepackage{enumitem}
\usepackage{xparse}
\usepackage{wasysym}
\usepackage{thm-restate}
\usepackage{wrapfig} 

\usepackage{amsthm}
\usepackage{thmtools}

\makeatletter
\newcommand{\eqnum}{\leavevmode\hfill\refstepcounter{equation}\textup{\tagform@{\theequation}}}
\makeatother

\newcommand{\fulltetra}[1]{\mathbf\Delta^{#1}}
\newcommand{\match}[1]{\partial{\fulltetra #1}}
\newcommand{\spine}[1]{\mathsf{Sp}^{#1}}

\tikzset{
 tikzdefforeground/.style={
  rectangle,
  rounded corners,
  draw=black, thick,
  minimum height=2em,
  minimum width=5em,
  inner sep=5pt,
  },
 tikzdefbackground/.style={
  rectangle,
  rounded corners,
  draw=black, 
  minimum height=2em,
  minimum width=5em,
  inner sep=5pt,
  },
 tikzpropproject/.style={
  >->>,
  shorten >=0.2cm,
  shorten <=0.2cm, 
  thick,
  },
 tikzprojectforeground/.style={
  ->>,
  shorten >=0.2cm,
  shorten <=0.2cm, 
  line width = 1.1pt,
  preaction={draw, -, line width=5pt, white},
  },
 tikzincludeforeground/.style={
  right hook->,
  shorten >=0.2cm,
  shorten <=0.2cm, 
  line width = 1.1pt,
  preaction={draw, -, line width=5pt, white},
  },
 tikzpropprojectforeground/.style={
  >->>,
  shorten >=0.2cm,
  shorten <=0.2cm, 
  line width = 1.1pt,
  preaction={draw, -, line width=5pt, white},
  },
 tikzincludebackground/.style={
  right hook->,
  shorten >=0.2cm,
  shorten <=0.2cm, 
  line width = 0.5pt,
  },
 tikzprojectbackground/.style={
  ->>,
  shorten >=0.2cm,
  shorten <=0.2cm, 
  line width = 0.5pt,
  },
 tikzpropprojectbackground/.style={
  >->>,
  shorten >=0.2cm,
  shorten <=0.2cm, 
  line width = 0.5pt,
  },
 tikzshortarrow/.style={
  shorten >=0.2cm,
  shorten <=0.2cm, 
  thick,
  },
 tikzequiv/.style={
  <->,
  shorten >=0.2cm,
  shorten <=0.2cm, 
  line width = 0.8pt,
  sloped, 
  above,
  },
 tikzbackgroundequiv/.style={
  <->,
  shorten >=0.2cm,
  shorten <=0.2cm, 
  line width = 0.5pt,
  right,
  },
}

\newcommand{\Ob}{\mathsf{Ob}}
\newcommand{\Hom}{\mathsf{Hom}}
\newcommand{\hcond}{h}
\newcommand{\Comp}{\_ \circ \_}
\newcommand{\Penta}{\pentagon}

\newcommand{\Ids}{\mathsf{Id}}

\newcommand{\isIsoAKS}{\mathsf{isIso}_{\mathsf{AKS}}}

\newcommand{\isIso}{\mathsf{isIso}}

\newcommand{\isneut}{\mathsf{isNeutral}}

\newcommand{\hornfillers}[3]{\Lambda^#2_#3\mbox{-}\mathsf{fillers}(#1)}
\newcommand{\contrhornfilling}[3]{\mathsf{has}\mbox{-}\mathsf{contr}\mbox{-}\Lambda^#2_#3\mbox{-}\mathsf{filling}(#1)}

\NewDocumentCommand\insertComma{>{\SplitList{}}m}{{\ProcessList{#1}{\listthiscomma}}}
\newcommand{\listthiscomma}[1]{#1\let\listthiscomma\listnextcomma}
\newcommand{\listnextcomma}[1]{,{#1}}

\NewDocumentCommand\insertLess{>{\SplitList{}}m}{{\ProcessList{#1}{\listthisless}}}
\newcommand{\listthisless}[1]{#1\let\listthisless\listnextless}
\newcommand{\listnextless}[1]{<{#1}}

\newcommand{\bound}[2][x]{#1_{S | S \subsetneq \{{\insertComma{#2}}\}}}

\newcommand{\boundconcrete}[1]{x_{S | S \subsetneq #1}}

\newcommand{\indexlist}[2]{0 \leq {\insertLess{#2}} \leq {#1}}

\newcommand{\indexlisthorn}[3]{\indexlist{#1}{#2} ; \, {#3} \in \{\insertComma{#2}\}}

\newcommand{\tetratuple}[1]{x_{S | S \subseteq \{{\insertComma{#1}}\}}}

\newcommand{\tetratupleconcrete}[1]{x_{S | S \subseteq \{{#1}\}}}

\newcommand{\horntuple}[3][x]{#1_{S | S \subset \{{\insertComma{#2}}\}; \, \exists p \neq #3. p \not\in S}}

\newcommand{\horntupleconcrete}[2]{x_{S | S \subset \{0,\ldots,#1\}; \, \exists p \neq #2. p \not\in S}}

\renewcommand{\fbox}[1]{#1}

\newcommand{\sm}[1]{\Sigma \left(#1\right),}
\newcommand{\jdeq}{\equiv}      
\newcommand{\defeq}{\vcentcolon\equiv}  
\newcommand{\refl}{\mathsf{refl}}

\newcommand{\ct}{%
  \mathchoice{\mathbin{\raisebox{0.5ex}{$\displaystyle\centerdot$}}}%
             {\mathbin{\raisebox{0.5ex}{$\centerdot$}}}%
             {\mathbin{\raisebox{0.25ex}{$\scriptstyle\,\centerdot\,$}}}%
             {\mathbin{\raisebox{0.1ex}{$\scriptscriptstyle\,\centerdot\,$}}}
}

\newcommand{\UU}{\mathcal{U}}
\newcommand{\N}{\mathbb{N}}

\newcommand{\iscontr}{\mathsf{isContr}}
\newcommand{\isprop}{\mathsf{isProp}}
\newcommand{\isset}{\mathsf{isSet}}

\def\compare#1#2#3#4{\if#1#3\if#2#41\else0\fi\else0\fi}

\newcommand{\istype}[1]{
  \edef\a{\compare-2#1\empty\empty}
  \if\a1 \iscontr \else
  \edef\b{\compare-1#1\empty\empty}
  \if\b1 \isprop \else
  \edef\c{#1}
  \if0\c \isset \else
  \mathsf{is}\mbox{-}{#1}\mbox{-}\mathsf{type} \fi\fi\fi
}

\newcommand{\mapfunc}[1]{\mathsf{ap}_{#1}}

\newcommand{\boundary}{\partial\Delta}

\renewcommand{\C}{\mathbb C}

 \def\boxwidth{1.9cm}
 \def\boxheight{0.55cm}
 \def\boxhigherheight{0.8cm}

\newcommand{\hacktaildoublearrow}{$\rightarrowtail$ \hspace{-16pt} $\twoheadrightarrow$}

\begin{document}

\begin{abstract}
Category theory in homotopy type theory is intricate as categorical laws can only be stated ``up to homotopy'', and thus require coherences.
The established notion of a \emph{univalent category}~\citep{ahrens:rezk} solves this by considering only truncated types,
roughly corresponding to an ordinary category.
This fails to capture many naturally occurring structures, 
stemming from the fact that the naturally occurring structures in homotopy type theory are not ordinary, but rather \emph{higher} categories.

Out of the large variety of approaches to higher category theory that mathematicians have proposed, we believe that, for type theory, the \emph{simplicial} strategy is best suited.
Work by \citet{lurie2009classification} and \citet{harpaz:quasi} motivates the following definition.
Given the first $(n+3)$ levels of a semisimplicial type $S$, we can equip $S$ with three properties: first, contractibility of the types of certain horn fillers; second, a completeness property; and third, a truncation condition.
We call this a \emph{complete semi-Segal $n$-type}.
This is very similar to an earlier suggestion by \citet{schreiber_nlab_inf1cat}.

The definition of a univalent ($1$-) category in~\citep{ahrens:rezk} can easily be extended or restricted to the definition of a \emph{univalent $n$-category} (more precisely, $(n,1)$-category) for $n \in \{0,1,2\}$, and we show that the type of complete semi-Segal $n$-types is equivalent to the type of univalent $n$-categories in these cases.
Thus, we believe that the notion of a complete semi-Segal $n$-type can be taken as the definition of a univalent $n$-category.

We provide a formalisation in the proof assistant Agda using a completely explicit representation of semi-simplicial types for levels up to $4$.
\end{abstract}

\maketitle

\newtheorem{theorem}{Theorem}[section]
\newtheorem{lemma}[theorem]{Lemma}
\newtheorem*{auxlemma*}{Auxiliary Lemma}
\newtheorem{corollary}[theorem]{Corollary}
\theoremstyle{definition}
\newtheorem{definition}[theorem]{Definition}
\newtheorem{example}[theorem]{Example}
\theoremstyle{remark}
\newtheorem{remark}[theorem]{Remark}

\section{Introduction}

The importance of category theory in all of mathematics and computer science can hardly be overestimated: it is a powerful tool that captures important overarching ideas in an elegant, concise framework, and allows one to work in high generality when establishing fundamental basic results. However, in the recently proposed foundational system for mathematics known as \emph{homotopy type theory} (HoTT), see \citep{hott-book}, a variant of Martin-L\"of's intensional type theory based on the observation that types can be viewed as spaces~\citep{awo-war:hott}, the application of categorical notions and ideas has so far proved quite a challenge. The main difficulties are related to the ubiquity of higher dimensional structures in the theory; since weak higher groupoids are in some sense taken as primitive concepts, naturally occurring categorical structures also tend to be higher dimensional, making ordinary category theory of somewhat limited applicability in this setting.

It has to be remarked that ordinary category theory in HoTT is not problematic by itself. In fact, it can be reproduced there quite efficiently, and doing so actually produces some substantial conceptual benefits compared to a traditional set-theoretic presentation; for example, it allows to give a precise form to the idea of ``invariance under equivalence'' that is so prominent in category theory. This has been achieved by \citet{ahrens:rezk}, where the authors adapt the familiar definition of category to the setting of HoTT. Their \emph{univalent categories} capture the same examples and enjoy the same properties as ordinary categories, and the usual results can be reproduced. To make this possible, morphisms of a univalent category are required to have ``homotopically trivial'' higher structure, which in HoTT parlance means that they are \emph{sets}. Consequently, equations about morphisms are really properties, and not additional structure, exactly like in the traditional setting.

Unfortunately --- and unsurprisingly, given the higher dimensional nature of HoTT --- univalent categories fail to capure many of the common examples that occur naturally when working within the system, egregiously including the universe of types $\UU$: if we take morphisms to be simply functions, then morphisms do not form sets, as function types can have as much higher dimensional structure as their codomain type. Since we have to give up the idea that the algebraic laws of a category should be mere properties, they have to be turned into further structure, which is itself subject to laws, and so on, ad infinitum. Therefore we naturally find ourselves in need of a notion of \emph{higher category}.

The idea of higher categories is of course not new to HoTT. However, the existing definitions and frameworks have proved quite hard to translate into HoTT, for essentially the same reasons that higher categories are so unavoidable in the first place: the basic building blocks of the theory are already equipped with higher dimensional categorical structure. In fact, all of the established approaches for dealing with the combinatorics of higher categories exploit, in one way or another, the fact that ultimately every mathematical object can be built out of sets, and those have no higher structure themselves. This assumption is certainly not validated by HoTT, since types such as the universe are not assumed to be in any way constructible using sets only.
One version of higher categories was suggested by \citet{cranch:concrete}, however their \emph{concrete categories} do not lead to a precise definition but rather to a collection of naturally occurring examples such as, again, type universes.

To appreciate the difficulties involved in expressing higher categorical notions in HoTT, it is important to understand the vast amount of combinatorial complexity arising from their higher dimensional structure, even forgetting about type theory for just a moment. Let us informally think of an $n$-category as a structure equipped with objects, morphisms between objects, 2-morphisms between morphisms, 3-morphisms between 2-morphisms, up to some level $n$, which could be infinity. Just like an ordinary category, we want an operation that allows us to compose morphisms, but this time we want the composition operation to only be associative up to an ``invertible'' 2-morphism. This means that we turned associativity from a \emph{mere property} of the composition operation into a \emph{structure} that returns a 2-morphism of a specific type given three composable morphisms as input. A similar process applies to identities and their laws, and to all forms of compositions and identities at every level. And this is not quite enough, since now all the laws that we turned into structure need laws of their own, which are referred to as \emph{higher coherences}. Of course, higher coherences require further higher coherences, until we reach the highest possible level $n$, at which point we revert back to laws. If $n$ is actually infinity, this process never ends, leaving us with infinitely many levels of coherence to manage.

One might think about suppressing coherences, but then the resulting structures are ill-behaved in subtle ways, and certain constructions (the simplest example of which being that of a slice category) will simply fail to work. Trying to make this process precise exactly in the form described above is easily seen to be unfeasible. The coherence properties one has to come up with to ensure that the overall notion is well-behaved do not follow any immediately apparent pattern, and their complexity grows fast enough to make a direct definition $n$-categories completely unworkable for $n$ as low as $4$. The solution is to organise this enormous amount of data in clever ways as to create simpler patterns that can then be more easily extended to arbitrary $n$ or to infinity.
Various strategies have been explored and lead to general definitions of higher categories, often not yet known to be equivalent.  They include opetopic~\citep{bae-dol:opetopes}, type theoretic~\citep{finsterMimram_weakCats}, operadic~\citep{batanin:omega-cat, leinster:higher}, simplicial~\citep{street:orientals, verity:complicial}, multisimplicial~\citep{tamsamani:weak-cat}, and cellular~\citep{joyal:disks} approaches.

The higher categories we are particularly interested in are the so-called $(n,1)$-categories,
for which an explicit survey has been given by \citet{Bergner2010}.
These are structures where all $k$-morphisms are invertible for $k > 1$. The reason for this restriction is that we can reuse the native higher structure of types given by equality (which is indeed invertible), and therefore only focus on objects, the first level of morphisms, and their composition, identities and coherences. Nevertheless, we still have a lot of data to manage, and for arbitrary or infinite $n$, it is still an open problem to assemble it in such a way as to make it possible to express the entire tower of coherences completely internally, at least without either restricting types to be sets, or extending the theory somehow --- cf.\ HTS \citep{voe:hts}, two-level type theory \citep{alt-cap-kra:two-level, ann-cap-kra:two-level}, and FOLDS \citep{dim:folds}.

Nevertheless, we can take inspiration from certain set-based models, in particular those based on \emph{simplicial sets} and \emph{simplicial spaces}, and obtain, for each fixed externally chosen natural number $n$, a notion of $(n,1)$-category which can be stated internally. A \emph{simplicial set} is an abstract way to describe a configuration of points, lines between these points, triangles, tetrahedra, and so on.
They come with very intuitive structure, namely so-called \emph{face maps} and \emph{degeneracy maps}.
The former allow us to get faces: given a tetrahedron, we can get any of the triangles in its ``boundary'', from a line, we can get its endpoints, and so on.
The degeneracy maps allow us to view a point as a trivial line, a line as a trivial triangle (where one of the two endpoints is duplicated), and so on.
Face and degeneracy maps fulfil intuitive laws, e.g.\ degenerating a triangle and taking one particular face leads to the original triangle.
Simplicial \emph{spaces}, which are just like simplicial sets, but where sets have been replaced by some notion of ``spaces''\footnote{Confusingly, these ``spaces'' are usually taken to be simplicial sets.}, can then be equipped with certain conditions that turn out to encode higher categorical structure in an elegant way. This is the basic idea underlying \emph{complete Segal spaces}, one of the existing formulations of $(\infty,1)$-categories.

To adapt the formalism of Segal spaces to HoTT, we would like to simply replace spaces with general types. This is however not a straightforward process, since the conditions on face and degeneracy maps are usually assumed to hold \emph{strictly} (i.e.\ up to equality and not up to homotopy), and there is no way to express this directly in HoTT, as the only equality the theory has access to corresponds to homotopy. As it turns out, using ideas from the theory of \emph{Reedy fibrant diagrams} over inverse categories~\citep{reedy:model-categories,shulman:inverse-diagrams}, we can encode those conditions implicitly using dependent types, but we have to give up degeneracies to make this possible. These structures are called \emph{semisimplicial types}. A priori, taking degeneracies out cripples the resulting object irreparably, since degeneracies are used to encode the identity part of a categorical structure. However, \citet{lurie2009classification} and \citet{harpaz:quasi} have observed that, under an appropriate assumption called \emph{completeness}, a ``weak'' degeneracy structure can actually be recovered a posteriori.

For some natural number $n$, the notion of Segal space can then be emulated in HoTT by starting with a semisimplicial type restricted to $(n+2)$ levels, and requiring the type of lines to be an $(n-1)$-type.
We call this structure a \emph{complete semi-Segal $n$-type}, and the idea is that it defines a \emph{univalent $n$-category}.
A definition similar to ours has earlier and independently (but also motivated by \citet{harpaz:quasi}) been suggested by Schreiber on the nLab~\citep{schreiber_nlab_inf1cat}, where it is simply called an \emph{$(n,1)$-category}; see Remark~\ref{rem:flexible-main-definition} for variations.

At this point, a problem becomes visible: what does a definition of higher categories have to satisfy in order to be considered ``correct''?
Since no notion of higher category exists so far in HoTT, there is nothing we can compare it to.
The only established notion is the one of a univalent category by \citet{ahrens:rezk}, 
and it is easy enough to generalise it to the definition of a \emph{univalent $2$-category}, and straightforward to simplify it to get the notion of a \emph{poset}.
What we can thus do is comparing 
\begin{itemize}
 \item \emph{complete semi-Segal sets} with \emph{posets},
 \item \emph{complete semi-Segal $1$-types} with \emph{univalent categories}, and
 \item \emph{complete semi-Segal $2$-types} with \emph{univalent 2-categories}.
\end{itemize}
We construct equivalences for each pair, in a modular way, such that each equivalence is a direct extension of the previous one.

\begin{figure}[htbp]
\begin{tikzpicture}[x= 1cm , y = -1cm]

 \newcommand{\fixedparbox}[1]{\parbox[][\boxheight][c]{\boxwidth}{\centering #1}}

 \pgfdeclarelayer{back}    
 \pgfsetlayers{back,main}  

 \foreach \x in {0,1,2,3}{
  \foreach \y in {0,1,2,3}{
   \foreach \z in {0,1}{
    \coordinate (A\x\y\z) at (3.3*\x+1.6*\z, 3.5*\y - 1.7*\z); 
 }}}


 

\newcommand{\textforeground}[1]{\textbf{#1}}

 
 \node[tikzdefforeground] (B000) at (A000) 
 {\fixedparbox{
   \textforeground{1-restr sSt} 
 }};

 \node[tikzdefforeground] (B010) at (A010) 
 {\fixedparbox{
   \textforeground{2-restr sSt} 
 }};

 \node[tikzdefforeground] (B020) at (A020) 
 {\fixedparbox{
   \textforeground{3-restr sSt} 
 }};

 \node[tikzdefforeground] (B030) at (A030) 
 {\fixedparbox{
   \textforeground{4-restr sSt} 
 }};

 
 \node[tikzdefforeground] (B110) at (A110) 
 {\fixedparbox{
   \textforeground{2-restr sSt with deg} 
 }};

 \node[tikzdefforeground] (B120) at (A120) 
 {\fixedparbox{
   \textforeground{3-restr sSt with deg} 
 }};

 \node[tikzdefforeground] (B130) at (A130) 
 {\fixedparbox{
   \textforeground{4-restr sSt with deg} 
 }};

 
 \node[tikzdefforeground] (B220) at (A220) 
 {\fixedparbox{
   \textforeground{3-restr univ sSt} 
 }};

 \node[tikzdefforeground] (B230) at (A230) 
 {\fixedparbox{
   \textforeground{4-restr univ sSt} 
 }};

 
 \node[tikzdefforeground] (B310) at (A310) 
 {\fixedparbox{
   \textforeground{univ sS set} 
 }};

 \node[tikzdefforeground] (B320) at (A320) 
 {\fixedparbox{
   \textforeground{univ sS 1-type} 
 }};

 \node[tikzdefforeground] (B330) at (A330) 
 {\fixedparbox{
   \textforeground{univ sS 2-type} 
 }};


\newcommand{\textbackground}[1]{#1}
 
 
 \node[tikzdefbackground] (B001) at (A001) 
 {\fixedparbox{
   \textbackground{graph} 
 }};

 \node[tikzdefbackground] (B011) at (A011) 
 {\fixedparbox{
   \textbackground{trans graph} 
 }};

 \node[tikzdefbackground] (B021) at (A021) 
 {\fixedparbox{
   \textbackground{wild semicat} 
 }};

 \node[tikzdefbackground] (B031) at (A031) 
 {\fixedparbox{
   \textbackground{wild 2-semicat} 
 }};

 
 \node[tikzdefbackground] (B111) at (A111) 
 {\fixedparbox{
   \textbackground{refl-trans graph} 
 }};

 \node[tikzdefbackground] (B121) at (A121) 
 {\fixedparbox{
   \textbackground{wild precat} 
 }};

 \node[tikzdefbackground] (B131) at (A131) 
 {\fixedparbox{
   \textbackground{wild 2-precat} 
 }};

 
 \node[tikzdefbackground] (B221) at (A221) 
 {\fixedparbox{
   \textbackground{wild category} 
 }};

 \node[tikzdefbackground] (B231) at (A231) 
 {\fixedparbox{
   \textbackground{wild univ 2-category} 
 }};
 
 
 \node[tikzdefbackground] (B311) at (A311) 
 {\fixedparbox{
   \textbackground{poset} 
 }};

 \node[tikzdefbackground] (B321) at (A321) 
 {\fixedparbox{
   \textbackground{univalent category} 
 }};

 \node[tikzdefbackground] (B331) at (A331) 
 {\fixedparbox{
   \textbackground{univalent 2-category} 
 }};
 

 
 \draw[tikzprojectforeground] (B010) to node {} (B000);
 \draw[tikzprojectforeground] (B020) to node {} (B010);
 \draw[tikzprojectforeground] (B030) to node {} (B020);
 
 \draw[tikzprojectforeground] (B120) to node {} (B110);
 \draw[tikzprojectforeground] (B130) to node {} (B120);
 
 \draw[tikzprojectforeground] (B230) to node {} (B220);
 
 \draw[tikzincludeforeground] (B310) to node {} (B320);
 \draw[tikzincludeforeground] (B320) to node {} (B330);
 
 
 \draw[tikzprojectforeground] (B110) to node {} (B010);
 \draw[tikzprojectforeground] (B120) to node {} (B020);
 \draw[tikzprojectforeground] (B130) to node {} (B030);
 
 \draw[tikzpropprojectforeground] (B220) to node {} (B120);
 \draw[tikzpropprojectforeground] (B230) to node {} (B130);
 
 \draw[tikzpropprojectforeground] (B310) to node {} (B110);
 
 \draw[tikzpropprojectforeground] (B320) to node {} (B220);
 \draw[tikzpropprojectforeground] (B330) to node {} (B230);

 \begin{pgfonlayer}{back}

 
 \draw[tikzequiv] (B000) to node {$\sim$} (B001);
 \draw[tikzequiv] (B010) to node {$\sim$} (B011);
 \draw[tikzequiv] (B020) to node {$\sim$} (B021);
 \draw[tikzequiv] (B030) to node {$\sim$} (B031);

 \draw[tikzequiv] (B220) to node {$\sim$} (B221);
 \draw[tikzequiv] (B230) to node {$\sim$} (B231);

 \draw[tikzequiv] (B110) to node {$\sim$} (B111);
 \draw[tikzequiv] (B120) to node {$\sim$} (B121);
 \draw[tikzequiv] (B130) to node {$\sim$} (B131);

 \draw[tikzequiv] (B310) to node {$\sim$} (B311);
 \draw[tikzequiv] (B320) to node {$\sim$} (B321);
 \draw[tikzequiv] (B330) to node {$\sim$} (B331);

 
 \draw[tikzprojectbackground] (B111) to node {} (B011);
 \draw[tikzprojectbackground] (B121) to node {} (B021);
 \draw[tikzprojectbackground] (B131) to node {} (B031);
 
 \draw[tikzpropprojectbackground] (B221) to node {} (B121);
 \draw[tikzpropprojectbackground] (B231) to node {} (B131);

 \draw[tikzpropprojectbackground] (B311) to node {} (B111);
 
 \draw[tikzpropprojectbackground] (B321) to node {} (B221);
 \draw[tikzpropprojectbackground] (B331) to node {} (B231);


 \draw[tikzprojectbackground] (B011) to node {} (B001);
 \draw[tikzprojectbackground] (B021) to node {} (B011);
 \draw[tikzprojectbackground] (B031) to node {} (B021);

 \draw[tikzprojectbackground] (B231) to node {} (B221);

 \draw[tikzprojectbackground] (B121) to node {} (B111);
 \draw[tikzprojectbackground] (B131) to node {} (B121);

 \draw[tikzincludebackground] (B311) to node {} (B321);
 \draw[tikzincludebackground] (B321) to node {} (B331);

 \end{pgfonlayer}
\end{tikzpicture}

\caption{
\emph{An overview of our structures (before considering completeness) and the connections between them.} See Sections~\ref{sec:composition-and-horns} and~\ref{sec:identity-and-degeneracy}. 
\\
Note: {\textbf{restr}} = restricted; {\textbf{sSt}} = semi-Segal type; {\textbf{univ}} = univalent; {\textbf{sS}} = semi-Segal [n-type]; {\textbf{trans}} = transitive; {\textbf{refl}} = reflexive.
\\
The front face of this 3D diagram consists of semisimplicial types with structure: 
in the front left column, we start with $(A_0, A_1)$ at the top
and add $A_2$, $A_3$, $A_4$ step by step.
Going from left to right on the front face, we add degeneracies, univalence, and a truncation condition (in the case of posets, we add them simultaneously).
\\
The back face of the diagram consists of categorical structures presented in ``ordinary'' style.
The back left column starts with graphs and adds composition structure in the first step, associativity in the next, and coherence for associativity (the pentagon) in the last step.
From left to right, the back face first adds identity structure, then a univalence condition, and finally a truncation condition.
\\
An arrow $A \twoheadrightarrow B$ means that definition $A$ arises from definition $B$ by adding one or more components (i.e.\ we can pass from $A$ to $B$ by forgetting something).
We write 
\hacktaildoublearrow 
instead of $\twoheadrightarrow$ to indicate that this added component is a proposition.
An arrow $A \mathrel{\mathrlap{\hspace{2.4pt}\raisebox{4pt}{$\scriptstyle{\sim}$}}\mathord{\leftrightarrow}} B$ expresses that the types $A$ and $B$ are equivalent.
An arrow $A \hookrightarrow B$ means that from an element of $A$ we can construct an element of $B$ in a canonical way.
} \label{fig:low-dimensional-categories}
\end{figure}

\begin{figure}[htbp]
\begin{tikzpicture}[x=0.995cm,y=-1cm]

 \newcommand{\fixedparbox}[1]{\parbox[][\boxheight][c]{\boxwidth}{\centering #1}}
 \newcommand{\higherfixedparbox}[1]{\parbox[][\boxhigherheight][c]{\boxwidth}{\centering #1}}

 \foreach \x in {0,1,2,3}{
   \coordinate (A\x) at (4*\x,0); 
 }

 \coordinate (CompCoor) at (1.5*4,2.5);
 \coordinate (CsstCoor) at (3*4,2.5);

\newcommand{\textbackground}[1]{#1}
 
 \node[tikzdefbackground] (B0) at (A0) 
 {\fixedparbox{
   \textbackground{wild 2-semicat} 
 }};

 \node[tikzdefbackground] (B1) at (A1) 
 {\fixedparbox{
   \textbackground{wild 2-precat} 
 }};

 \node[tikzdefbackground] (B2) at (A2) 
 {\fixedparbox{
   \textbackground{wild univ 2-category} 
 }};
 
 \node[tikzdefbackground] (B3) at (A3) 
 {\fixedparbox{
   \textbackground{univalent 2-category} 
 }};
 
 \draw[tikzprojectbackground,above] (B1) to node {\scriptsize $\pm$ identities} (B0);
 \draw[tikzpropprojectbackground,above] (B2) to node {\scriptsize $\pm$ univalence} (B1);
 \draw[tikzpropprojectbackground,above] (B3) to node {\scriptsize $\pm$ truncation} (B2);

 \draw[tikzpropprojectbackground, bend right=20] (B3) to node {} (B0);

 \node[tikzdefbackground] (Comp) at (CompCoor) 
 {\fixedparbox{
   \textbackground{wild compl 2-semicat} 
 }};

 \node[tikzdefbackground] (Csst) at (CsstCoor) 
 {\fixedparbox{
   \textbackground{complete 2-semicat} 
 }};

 \draw[tikzpropprojectbackground,sloped,above] (Comp) to node {\scriptsize $\pm$ completeness} (B0);
 \draw[tikzpropprojectbackground,above] (Csst) to node {\scriptsize $\pm$ truncation} (Comp);
 \draw[tikzbackgroundequiv] (B2) to node {} (Comp);
 \draw[tikzbackgroundequiv] (B3) to node {$\sim$} (Csst);

 
\end{tikzpicture}

\caption{%
\emph{Connection between univalence and completeness.} See Section~\ref{sec:completeness}. \\
The upper row is taken from Figure~\ref{fig:low-dimensional-categories} and shows that the definition of a univalent 2-category is obtained by starting with a wild 2-category, adding identities, then univalence, and finally a truncation condition.
Arrows are drawn with the same convention in mind as explained in the caption of Figure~\ref{fig:low-dimensional-categories}.
The long top arrow indicates that there is at most one (univalent) 2-categorical structure for a given wild 2-semicategory by Theorem~\ref{thm:id-structure-unique}.
The bottom row outlines an alternative construction: starting with a wild 2-semicategory, we can add completeness and truncation.
The resulting \emph{complete 2-semicategories} are equivalent to 2-categories. In particular, the identity structure can be constructed.
} \label{fig:univalence-completeness}
\end{figure}

This construction is split into several steps, and nearly all parts proceed \emph{without} making any truncatedness assumption.
We thereby obtain equivalences at a high generality between ``ill-behaved'' structures which we call \emph{wild}, although they might as well be called \emph{incoherent} or \emph{not necessarily coherent}.
First, we establish an equivalence between wild semicategorical structures and semi-Segal types.
Second, we show that equipping wild semicategories with identity structure (including coherences) amounts exactly to equipping semi-Segal types with degeneracies.
Third, we can add a univalence condition and fourth, a truncation condition, both of which are of course propositions.
Until here, our work is summarised by Figure~\ref{fig:low-dimensional-categories} on page~\pageref{fig:low-dimensional-categories}.

It is worth noting that the type ``having a degeneracy structure'' (equivalently, ``having an identity structure'') is in general \emph{not} a proposition, but we show that it is as soon as the structure is sufficiently truncated (e.g.\ for $2$-semicategories, the type of morphism must be a $1$-type).
This is important to make the connection to the next part of the paper, which is devoted to \emph{completeness}.
We define the completeness property, and show that it allows us to construct a degeneracy structure.
We show that a semi-Segal type is complete if and only if it is univalent, which can only be formulated assuming that the semi-Segal type already comes with degeneracies.
One instance of this result is presented in Figure~\ref{fig:univalence-completeness} (page~\pageref{fig:univalence-completeness}).
Together with the mentioned lemma that degeneracy structure is unique in the truncated case, we can conclude that complete semi-Segal $n$-types are equivalent to univalent $n$-categories for $n \in \{0,1,2\}$.
Whether this statement can be formulated and proved for general $n$ is unclear to us, but we offer a brief discussion in the final section (conclusions).

\paragraph{\textbf{Main contribution} (see Definition~\ref{def:final-definition} and Theorem~\ref{thm:main-result})}
We define the notion of a \emph{complete semi-Segal $n$-type} and show that, for $n \in \{0,1,2\}$, it is equivalent to the type of \emph{univalent $n$-categories}.
This suggests that we can take it as a definition for \emph{univalent $n$-category}, a definition of which has in homotopy type theory so far only been given for very small $n$. 

\paragraph{\textbf{Agda formalisation}}
We provide a formalisation%
\footnote{This formalisation is available at \url{https://gitlab.com/pcapriotti/agda-segal}.}
in the proof assistant Agda~\citep{norell:towards}, which proves 
the equivalences summarised in Figure~\ref{fig:low-dimensional-categories}.

\paragraph{\textbf{Organisation}}
Section~\ref{sec:preliminaries} specifies the theory we work in and recalls the definition of \emph{univalent categories} and the idea of \emph{semisimplicial types}.
In Section~\ref{sec:composition-and-horns}, we construct the equivalence between wild semicategorical structures and semi-Segal types.
These are equipped with identity and degeneracy structures (respectively) in Section~\ref{sec:identity-and-degeneracy}, and several lemmata about them are shown.
In Section~\ref{sec:completeness}, we introduce completeness and its prove its consequences.
Section~\ref{sec:conclusions} summarises our work and formulates the main result, and discusses applications and consequences of our work.

\section{Type Theory, Univalent Categories, and Semisimplicial Types} \label{sec:preliminaries}

Our work takes place in homotopy type theory, and we want to use the current section to clarify which theory we are working in (Section~\ref{sec:hott}). 
Moreover, as our work is related to univalent categories, we review the original construction by \citet{ahrens:rezk} in Section~\ref{sec:unicat}.
The purpose is solely to provide some context, as the actual definition of a univalent category will arise naturally from the constructions in the main part of this work. 
A further central concept that we need are \emph{semisimplicial types}~\citep{uf:semisimplicialtypes}.
We review the idea and introduce some notation and terminology in Section~\ref{sec:semisimplicialtypes}.

\subsection{Homotopy type theory} \label{sec:hott}

The theory we work in is a (slight notational variant of) the standard version of homotopy type theory, as presented in~\citep{hott-book}.  We make use of all the basic type formers, such as $\Pi$, $\Sigma$, equality and unit types, plus a \emph{univalent} universe.
However, we do not make use of higher inductive types, or even ordinary inductive types, since our development is completely elementary and limited to low, fixed restriction levels.  In particular, we do not use truncation operators such as propositional or set truncation.

In the following, we will call a type of the form $\Sigma (x : A), x = y$ a \emph{singleton}.  It is an immediate consequence of path induction for equality types that singletons are contractible types, and this is something that will be used repeatedly throughout the paper.
As a notational simplification, we omit the type annotation in $\Pi(x:A), B(x)$ if $A$ can easily be inferred, and we write $\Pi x, B(x)$ instead. 

\subsection{Univalent Categories} \label{sec:unicat}

Let us recall the definition of a \emph{univalent category}~\citep{ahrens:rezk} in homotopy type theory, which the authors present in two stages.
They first define a \emph{precategory}, and afterwards add a \emph{saturation} or \emph{univalence} condition, as follows:

\begin{definition}[precategory] \label{def:precategory}
 A \emph{precategory} is given by
 \begin{itemize}
  \item $\Ob : \UU$, a type of objects
  \item $\Hom : \Ob \times \Ob \to \UU$, a family of morphisms
  \item $\hcond : \Pi a b, \isset(\Hom(a,b))$, the condition that the types of morphisms are sets
  \item $(\Comp) : \Pi a b c, \Hom(b,c) \times \Hom(a,b) \to \Hom(a,c)$, the composition operation; we write $g \circ f$ instead of $(\Comp)(a,b,c,g,f)$
  \item $\alpha : \Pi fgh, h \circ (g \circ f) = (h \circ g) \circ f$, an equality certifying associativity (implicitly quantified over four objects)
  \item $\Ids : \Pi a, \Hom(a,a)$, the identity morphisms
  \item $\Pi (f : \Hom(a,b)), f \circ \Ids_a = f$ and $\Pi (f : \Hom(a,b)), \Ids_b \circ f = f$, the identity laws.
 \end{itemize}
\end{definition}

Recall that \citet{ahrens:rezk} go on and define the notion of an isomorphism in the straightforward way: 
a morphism $f : \Hom(a,b)$ is an isomorphism if it has an inverse, 
\begin{equation} \label{eq:isIsoAKS}
 \isIsoAKS(f) \defeq \sm{g : \Hom(b,a)} (g \circ f = \Ids_a) \times (f \circ g = \Ids_b).
\end{equation}
We write $a \cong_{\mathsf{AKS}} b$ for the type of isomorphisms $\sm{f : \Hom(a,b)} \isIsoAKS(f)$. 
There is a canonical function $\mathsf{idtoiso}_{\mathsf{AKS}} : a = b \to a \cong_{\mathsf{AKS}} b$, defined by path induction, where $\refl_a$ is sent to $\Ids_a$.
\begin{definition}[univalent category~{\cite[Def.~3.6]{ahrens:rezk}}] \label{def:univalent-category}
 A precategory is a \emph{univalent category} if, for all objects $a,b$, the function $\mathsf{idtoiso}_{\mathsf{AKS}}$ is an equivalence of types.
\end{definition}

\subsection{Semisimplicial Types} \label{sec:semisimplicialtypes}

A semisimplicial type restricted to level 3 as described in~\citep{uf:semisimplicialtypes} is given by a tuple $(A_0, A_1, A_2, A_3)$,
where:
\begin{itemize}
 \item $A_0$ can be thought of simply as a type of \emph{points}.
 \item Given any two such points $x_0, x_1 : A_0$, we have a type of \emph{edges} $A_1(x_0,x_1)$. Thus, $A_1$ is a family of types indexed over $A_0 \times A_0$.
 \item If we have three points $x_0, x_1, x_2 : A_0$, and three edges $x_{01} : A_1(x_0,x_1)$, $x_{12} : A_1(x_0,x_1)$, and $x_{02} : A_1(x_0,x_2)$, we can picture these six elements as an empty triangle. $A_2$ is a family of types indexed over such empty triangles, and we can think of elements of $A_2(x_0,x_1,x_2,x_{01},x_{12},x_{02})$ as \emph{fillers} for the triangle.
 \item Assume we are given four points, together with six edges, and four triangle fillers which fit together such that they form an empty tetrahedron.
 We think of $A_3$ as a family of types which, for any such empty tetrahedron, gives us a type of \emph{tetrahedron fillers}.
\end{itemize}
In fully explicit type-theoretic presentation, this means that the types are:
\begin{equation} \label{eq:3sstype-long}
\begin{alignedat}{3}
 & A_0 &&: &\; & \UU \\
 & A_1 &&: && A_0 \times A_0 \to \UU \\
 & A_2 &&: && \Sigma (x_0, x_1, x_2 : A_0), (x_{01} : A_1(x_0, x_1)), (x_{12} : A_1(x_1,x_2)), (x_{02} : A_1(x_0, x_2)) 
              \to \UU \\
 & A_3 &&: && \Sigma (x_0, x_1, x_2, x_3 : A_0), \\
 &     &&  && \phantom{\Sigma} (x_{01} : A_1(x_0, x_1)), (x_{12} : A_1(x_1,x_2)), (x_{23} : A_1(x_2, x_3)), \\
 &     &&  && \phantom{\Sigma} (x_{02} : A_1(x_0,x_2)), (x_{13} : A_1(x_1,x_3)), (x_{03} : A_1(x_0,x_3)), \\
 &     &&  && \phantom{\Sigma} (x_{012} : A_2(x_0, x_1, x_2, x_{01}, x_{12}, x_{02})), (x_{013} : A_2(x_0, x_1, x_3, x_{01}, x_{13}, x_{03})), \\
 &     &&  && \phantom{\Sigma} (x_{023} : A_2(x_0, x_2, x_3, x_{02}, x_{23}, x_{03})), (x_{123} : A_2(x_1, x_2, x_3, x_{12}, x_{23}, x_{13}))  \\
 &     &&  && \to \UU
\end{alignedat}
\end{equation}
These type expressions are uniform, but rather long (especially the type of $A_3$).
We therefore use the following shorthand notation.
First, instead of $x_{023} : A_2(x_0, x_2, x_3, x_{02}, x_{23}, x_{03})$, we write  $x_{023} : A_2(\bound{023})$.
Thus, the symbol $\subsetneq$ should be read as \emph{proper nonempty subset}.
Second, instead of writing a list $x_{012}: A_2(\ldots), \ldots, x_{123}:A_2(\ldots)$ with their types explicitly as in~\eqref{eq:3sstype-long}, we write $(x_{ijk} : A_2(\bound{ijk}))_{\indexlist{3}{ijk}}$.
In this notation,~\eqref{eq:3sstype-long} becomes:
\begin{alignat}{3} \label{eq:3sstype-short}
 & A_0 &&: &\; & \UU \\
 & A_1 &&: && \Sigma (x_i : A_0)_{\indexlist 1 i} \to \UU \\
 & A_2 &&: && \Sigma (x_i : A_0)_{\indexlist 2 i}, (x_{ij} : A_1(\bound {ij}))_{\indexlist 2 {ij}} \to \UU \\
 & A_3 &&: && \Sigma (x_i : A_0)_{\indexlist 3 i}, (x_{ij} : A_1(\bound {ij}))_{\indexlist 3 {ij}}, (x_{ijk} : A_2(\bound {ijk}))_{\indexlist 3 {ijk}} \to \UU.
 \label{eq:3d-tetra} \\
\intertext{Thanks to the simplified notation, it is feasible to write down the next stage, namely the type of 4-dimensional tetrahedron fillers:}
 & A_4 &&: &\; & \Sigma (x_i : A_0)_{\indexlist 4 i}, (x_{ij} : A_1(\bound {ij}))_{\indexlist 4 {ij}}, \notag \\
 &     &&  &   & \phantom{\Sigma} (x_{ijk} : A_2(\bound {ijk}))_{\indexlist 4 {ijk}}, (x_{ijkl} : A_3(\bound {ijkl}))_{\indexlist 4 {ijkl}} \label{eq:4d-tetra} \\
 &     &&  &   & \to \UU. \notag
\end{alignat}

\begin{definition}[restricted semisimplicial type] \label{def:restr-sst}
 For $n \in \{0,1,2,3,4\}$, we say that an \emph{n-restricted semisimplicial type} is a tuple $(A_0, \ldots, A_n)$
 with the types as given in~\eqref{eq:3sstype-short} and~\eqref{eq:4d-tetra}.
\end{definition}

\begin{remark} \label{rem:hott-is-too-weak}
 While it is perfectly possible to take a concrete (externally fixed) natural number $n$ and generate a type-theoretic expression $A_n$ which represents a type of $n$-dimensional tetrahedra fillers, it is (as already mentioned in the introduction) a long-standing open problem whether this can be encoded \emph{internally} in type theory.
 Concretely, it is unknown whether one can define a ``classifier'' $S : \N \to \UU$ such that for each $n$ the elements of $S(n)$ correspond to tuples $(A_0, \ldots, A_n)$.
 This is the reason why Definition~\ref{def:restr-sst} does not define \emph{$n$-restricted semisimplicial types} for any $n : \N$.
 We could do it for an externally fixed number $n$; however, $n \leq 4$ is sufficient for the main part of the paper.
 Thus, we have chosen the rather modest formulation in Definition~\ref{def:restr-sst} to ensure that the paper stays close to our formalisation.
\end{remark}

We use the remainder of this subsection to explain notation and terminology related to semisimplicial types.
Above, we have talked about \emph{triangle fillers} and \emph{(higher) tetrahedron fillers}
when referring to elements of $A_2(\ldots)$, $A_3(\ldots)$, or $A_4(\ldots)$.
Sometimes, it is useful to talk about the corresponding \emph{total spaces}, i.e.\ the actual types of triangles or (higher) tetrahedra.
As it will always be clear which semisimplicial type $(A_0, \ldots, A_n)$ we refer to (with $n \leq 4$), we simply write $\fulltetra n$ for the type of $n$-dimensional triangles/tetrahedra, slightly deviating from the usual convention where $\Delta^n$ denotes the standard $n$-simplex.
The type $\fulltetra n$ is a nested $\Sigma$-type with $2^n -1$ components, one for each non-empty subset of $\{0, \ldots, n\}$.
We can write this down concretely:
\begin{equation}
\begin{alignedat}{3}
 & \fulltetra 0 &\; & \defeq &\; \; & A_0 \qquad \text{-- or $(x_0 : A_0)$, to give the point a name}\\
 & \fulltetra 1 && \defeq && \Sigma (x_i : A_0)_{\indexlist 1 i}, (x_{01} : A_1(\bound{01})) \\
 & \fulltetra 2 && \defeq && \Sigma (x_i : A_0)_{\indexlist 2 i}, (x_{ij} : A_1(\bound{ij})_{\indexlist 2 {ij}}, (x_{012} : A_2(\bound{012})) \\
 & \fulltetra 3 && \defeq && \Sigma (x_i : A_0)_{\indexlist 3 i}, (x_{ij} : A_1(\bound{ij})_{\indexlist 3 {ij}}, \\
 &              &&        && \phantom{\Sigma} (x_{ijk} : A_2(\bound{ijk})_{\indexlist{3}{ijk}}, (x_{0123} : A_3(\bound{0123})) \\
 & \fulltetra 4 && \defeq && \Sigma (x_i : A_0)_{\indexlist 4 i}, (x_{ij} : A_1(\bound{ij})_{\indexlist 4 {ij}}, \\
 &              &&        && \phantom{\Sigma} (x_{ijk} : A_2(\bound{ijk})_{\indexlist{4}{ijk}}, (x_{ijkl} : A_3(\bound{ijkl}))_{\indexlist{4}{ijkl}}, \\
 &              &&        && \phantom{\Sigma} (x_{01234} : A_4(\bound{01234}))
\end{alignedat}
\end{equation}

We can also consider the type of \emph{boundaries} for an $n$-dimensional tetrahedron.
We write $\match n$ for this type. 
It is given simply by the type of $\fulltetra n$ as stated above with the very last component (the \emph{filler}) removed.
Thus, $\match n$ is a nested $\Sigma$-type with $2^n-2$ components, one for each non-empty proper subset of $\{0,\ldots,n\}$.
With this notation, we can represent the type of $A_n$ of~(\ref{eq:3sstype-short}-\ref{eq:4d-tetra}) as 
\begin{equation} \label{eq:An-as-family-over-boundary}
 A_n : \match n \to \UU.
\end{equation}

For $1 \leq n \leq 4$ and $0 \leq i \leq n$, we can define the so-called \emph{face map} $d^{n}_i : \fulltetra n \to \fulltetra{n-1}$.
Intuitively, it gives us one of the $(n+1)$ faces, namely the one which is opposite to the vertex labeled with $i$, by projection.
For an element $x : \fulltetra n$, we can by definition assume that it is a tuple
 $x \; \jdeq \; (\tetratupleconcrete{0, \ldots, n})$.
The function $d^n_i$ discards every $x_S$ which has $i \in S$; that is, 
\begin{equation}
 d^n_i(x) \; \defeq \; (x_{S|S \subseteq \{0,\ldots, i-1, i+1, \ldots,n\}}).
\end{equation}

\begin{remark}
 Given a category $\mathcal C$, it is in general not known how to encode a type of \emph{strict} functors $\mathcal C \to \UU$ in type theory (one might try to state the functor laws using the internal equality type, but this will not give a well-behaved notion).
 Consider the special case where $\mathcal C$ is the category with five objects $[0]$, $[1]$, $[2]$, $[3]$ and $[4]$, with $[n] \defeq \{0,\ldots,n\}$ and where morphisms from $[m]$ to $[n]$ are strictly increasing functions from $[n]$ to $[m]$.
 We write this category as $(\Delta^{\leq 4}_+)^\mathsf{op}$.
 In this case, a $4$-restricted semisimplicial type does indeed encode a strict functor $A : (\Delta^{\leq 4}_+)^\mathsf{op} \to \UU$: 
 We let $A([n]) \defeq \fulltetra n$. 
 To define $A$ on morphisms, it is enough to consider morphisms from $[n]$ to $[n-1]$ which are given by omitting a number $i$, since any morphism can be written as a composition of such maps; and $A$ maps such a morphism to $d^n_i$. 
 Since the $d^n_i$ are projections, one can see easily that the functor laws hold judgmentally.
 In general, this encoding works for any finite category $\mathcal C$ which has no nontrivial ``cycles''~\citep{shulman:inverse-diagrams}.
 A diagram $A$ of this form is usually called \emph{Reedy fibrant} \citep{reedy:model-categories} and $\boundary^n$ is called a \emph{matching object} of $A$.
\end{remark}

\section{Composition Structure and Horn Fillers} \label{sec:composition-and-horns}

The first parts of the categorical and higher categorical structures that we consider are notions of \emph{composition}.
Our various structures come in two different presentations, and with several levels of well-behavedness.

\subsection{Wild Semicategories}

Let us begin with categorical structures presented in the style of the precategories in Definition~\ref{def:precategory}. 
In fact, all of the following notions can be understood as weak versions of precategories:

\begin{definition}[wild semicategorical structure] \label{def:graphs-1-to-4}
 We define the following, where each step adds one level of structure:
 \begin{enumerate}
  \item A \emph{graph} is a type $\Ob$ together with a family $\Hom : \Ob \times \Ob \to \UU$.
  \item A \emph{transitive graph} is a graph, together with a \emph{composition operator}   \label{item:graphs-1-to-4-snd}
   \begin{equation} (\Comp) : \Pi a b c, \Hom(b,c) \times \Hom(a,b) \to \Hom(a,c); \end{equation} we write $g \circ f$ instead of $(\Comp)(a,b,c,g,f)$.
  \item A \emph{wild semicategory} is a transitive graph which, in addition, has an \emph{associator} \label{item:graphs-1-to-4-thrd}
  \begin{equation} \alpha : \Pi fgh, h \circ (g \circ f) = (h \circ g) \circ f.\end{equation}
  \item A \emph{wild 2-semicategory} is a wild semicategory together with a \emph{pentagonator},     \label{item:graphs-1-to-4-frth}
  \begin{equation} \label{eq:type-of-penta}
  \begin{alignedat}{2}
   &\Penta : \Pi fghk, 
     && 
     \mapfunc {k \circ \_} (\alpha(f,g,h)) \ct 
     \alpha(g \circ f, h, k) \ct
     \mapfunc {\_ \circ f} (\alpha(g,h,k)) \\ 
     && \; = \; & 
     \alpha(f, h \circ g, k) \ct
     \alpha(f,g,k\circ h).
  \end{alignedat}
  \end{equation}
 \end{enumerate}
\end{definition}

\begin{remark}
The type of $\Penta$, as given in~\eqref{eq:type-of-penta}, can alternatively be described as the type of proofs of commutativity for the following pentagon:

\begin{equation}
\begin{tikzpicture}[x=3cm,y=-1.5cm,baseline=(current bounding box.center)]
\node (P1) at (1,2) {\fbox{$((k \circ h) \circ g) \circ f$}}; 
\node (P2) at (-1,2) {\fbox{$(k \circ (h \circ g)) \circ f$}}; 
\node (P3) at (-1,1) {\fbox{$k \circ ((h \circ g) \circ f)$}}; 
\node (P4) at (0,0.2) {\fbox{$k \circ (h \circ (g \circ f))$}}; 
\node (P5) at (1,1) {\fbox{$(k \circ h) \circ (g \circ f)$}}; 

\draw[<-, shorten >=0.1cm, shorten <=0.1cm, thick] (P1) to node [above] {\scriptsize $\mapfunc {\_ \circ f} (g,h,k)$} (P2);
\draw[<-, shorten >=0.1cm, shorten <=0.1cm, thick] (P2) to node [left] {\scriptsize $\alpha(g \circ f, h, k)$} (P3);
\draw[<-, shorten >=0.1cm, shorten <=0.1cm, thick] (P1) to node [right] {\scriptsize $\alpha(f,g,k\circ h)$} (P5);
\draw[<-, shorten >=0.1cm, shorten <=0.1cm, thick] (P5) to node [above right] {\scriptsize $\alpha(f, h \circ g, k)$} (P4);
\draw[->, shorten >=0.1cm, shorten <=0.1cm, thick] (P4) to node [above left] {\scriptsize $\mapfunc {k \circ \_} (\alpha(f,g,h))$} (P3);
\end{tikzpicture}
\end{equation}
\end{remark}

\subsection{Semi-Segal Types}

As outlined in the introduction, one of our goals is to show how semisimplicial types enable us to encode categorical structure. 
The idea is that $A_0$ (the type of points) will form the type of objects, and $A_1(x_i, x_j)$ (the type of edges) will form the type of morphisms between $x_i$ and $x_j$.
We can require that $A_1$ is a family of \emph{sets}, as it is the case for the univalent categories by \citet{ahrens:rezk}. 

For the rest, we have to add more structure.
In the current section, we want to discuss the structure which is necessary to encode a composition operation (based on $A_2$), which may come with an associativity operator (based on $A_3$) and a pentagonator (coherence for associativity, based on $A_4$).
The notion of a \emph{horn} becomes important here.
The type of horns is similar to $\match n$, but with one additional component removed.
Thus, a horn is indexed over two natural numbers, say $n$ and $m$, where for $0 \leq m \leq n \leq 4$, and it represents the type of $n$-dimensional tetrahedra without the single cell of dimension $n$ and without the face opposite to $x_m$.
For example, 
for $n \jdeq 2$ and $m \jdeq 1$, such a horn consists of three points $x_0, x_1, x_2 : A_0$ and two edges $x_{01} : A_1(x_0, x_1)$, $x_{12} : A_1(x_1,x_2)$, thus we write
\begin{equation}
 \Lambda^2_1 \defeq \Sigma (x_0,x_1,x_2 : A_0), A_1(x_0,x_1) \times A_1(x_1,x_2). 
\end{equation}
In our shorthand notation, this becomes
\begin{equation}
 \Lambda^2_1 \defeq \Sigma (x_i : A_0)_{\indexlist 2 i}, (x_{ij} : A_1(\bound {ij}))_{\indexlisthorn 2 {ij} 1}. 
\end{equation}
Similarly, a $\Lambda^3_m$-horn (for $m \in \{0,1,2,3\}$) consists of four points, six edges, and three triangle fillers (one triangle filler is missing, namely the one not containing $x_m$).
Let us record this:
\begin{definition}[horns]
 Given a 2-restricted semisimplicial type $(A_0, A_1, A_2)$,
 and $m \in \{0,1,2\}$, we define the \emph{type of $\Lambda^2_m$-horns} to be
 \begin{align}
  \Lambda^2_m \defeq  \Sigma & (x_i : A_0)_{\indexlist 2 i}, (x_{ij} : A_1(\bound {ij}))_{\indexlisthorn 2 {ij} m}. 
\intertext{If we have a 3-restricted semisimplicial type, i.e.\ some $A_3$ extending $(A_0, A_1, A_2)$, we allow $m \in \{0,1,2,3\}$ and define}
  \Lambda^3_m \defeq  \Sigma & (x_i : A_0)_{\indexlist 3 i},  \\  
                             & (x_{ij} : A_1(\bound {ij}))_{\indexlist 3 {ij}}  , \notag \\
                             & (x_{ijk} : A_2(\bound {ijk}))_{\indexlisthorn 3 {ijk} m}.  \notag 
\intertext{If we in addition have $A_4$, and any $m \in \{0,1,2,3,4\}$, we define}
   \Lambda^4_m \defeq  \Sigma & (x_i : A_0)_{\indexlist 4 i}, \\
        & (x_{ij} : A_1(\bound {ij}))_{\indexlist 4 {ij}}, \notag \\
        & (x_{ijk} : A_2(\bound {ijk}))_{\indexlist 4 {ijk} m}, \notag \\
        & (x_{ijkl} : A_3(\bound {ijkl}))_{\indexlisthorn 4 {ijkl} m}.\notag 
\end{align}
\end{definition}
In general, an element of $\Lambda^n_m$ is called an \emph{inner horn} if $0 < m < n$, and an \emph{outer horn} if $m = 0$ or $m = n$.

In explicit representation, an element of $\Lambda^n_m$ has $2^n-3$ components, namely one for every nonempty subset of $\{0, \ldots, n\}$ which lacks at least one number different from $m$.
Thus, we can assume that for example a $(3,1)$-horn $u : \Lambda^3_1$ is a tuple and write it as
\begin{equation}
 u \jdeq (\horntuple{0123}{1}).
\end{equation}
The ``missing'' bit in $u$ is a triangle filler $x_{023} : A_2(\bound{023})$, and if we are given such an $x_{023}$, we can consider the type of tetrahdron fillers $x_{0123} : A_3(\bound{0123})$.
We call the type of such pairs $(x_{023},x_{0123})$ the type of \emph{fillers} for the horn $u$.
In the following definition, we write $[n]$ for the set $\{0,1,\ldots,n\}$, and we write $[n]-m$ for the same set where the number $m$ is removed.

\begin{definition}[horn fillers]
 Assume we have a semisimplicial type $(A_0, \ldots, A_n)$, with $n \in \{2,3,4\}$.
 Given a horn $u : \Lambda^n_m$, where we can assume $u \jdeq (\horntupleconcrete{n}{m})$, the \emph{type of horn-fillers of $u$}
 is the type
 \begin{equation}
  \hornfillers u n m \defeq \Sigma \left(x_{[n]-m} : A_{n-1}(\boundconcrete{[n]-m})\right) A_n(\boundconcrete{[n]}).
 \end{equation}
 Note that $\hornfillers u n m$ is always a type of pairs of exactly two components.
 We say that the semisimplicial type $(A_0, \ldots, A_n)$ has \emph{contractible $(n,m)$-horn filling} if the type of fillers is contractible for any element of $\Lambda^n_m$,
 \begin{equation}
  \contrhornfilling{A_0, \ldots, A_n}{n}{m} \jdeq \Pi(u : \Lambda^n_m), \iscontr(\hornfillers u n m).
 \end{equation}
\end{definition}

There are a couple of alternative formulations of the following definition of semi-Segal types.
We will discuss them later in Remark~\ref{rem:on-semiSegal-definition}.
\begin{definition}[restricted semi-Segal types] \label{def:semisegal-1-to-4}
 For $n \in \{1, 2, 3, 4\}$, we define an \emph{n-restricted semi-Segal type} to be an n-restricted semisimplicial type which satisfies \begin{equation}
  \contrhornfilling{A_0, \ldots, A_n}{p}{1},
 \end{equation}
 for $p \in \{2,\ldots,n\} $. In detail, this means:
 \begin{enumerate}
  \item A \emph{1-restricted semi-Segal type} is the same as a 1-restricted semisimplicial type.
  \item A \emph{2-restricted semi-Segal type} has $A_0$, $A_1$, $A_2$, and $h_2 : \contrhornfilling{A_0, A_1, A_2}{2}{1}$.
  \item A \emph{3-restricted semi-Segal type} consists of $A_0$, $A_1$, $A_2$, $h_2$ as above, plus the components $A_3$ and $h_3 : \contrhornfilling{A_0, \ldots, A_3}{3}{1}$.
  \item A \emph{4-restricted semi-Segal type} has, in addition to the above, the type family $A_4$ and the component $h_4 : \contrhornfilling{A_0, \ldots, A_4}{4}{1}$.
 \end{enumerate}
\end{definition}

The connection of semi-Segal types with the semicategories discussed before is the following statement, the proof of which will be the subject of Section~\ref{sec:proof-of-equivalence-sst-wildcats}:

\begin{restatable}{theorem}{composition}
\label{thm:wild-semi}
 Definitions~\ref{def:semisegal-1-to-4} and~\ref{def:graphs-1-to-4} define the same structures:
 \begin{enumerate}
  \item The type of graphs is equivalent to the type of 1-restricted semi-Segal types.\label{thm:wild-semi:1}
  \item The type of transitive graphs is equivalent to the type of 2-restricted semi-Segal types.
  \item The type of wild semicategories is equivalent to the type of 3-restricted semi-Segal types.
  \item The type of wild 2-semicategories is equivalent to the type of 4-restricted semi-Segal types.
 \end{enumerate}
\end{restatable}

\subsection{Interlude: On Horns, Spines, and Tetrahedra}

Before giving the somewhat lengthy proof of Theorem~\ref{thm:wild-semi}, we want to show some simple but useful auxiliary lemmata.  Assume that $(A_0, \ldots, A_n)$ is an $n$-restricted semi-Segal type ($n \in \{2,3,4\}$ as before).

Recall that, so far, we have considered $\fulltetra n$, the type of full tetrahedra; $\match n$, the type of boundaries; and $\Lambda^k_i$, the type of $(k,i)$-horns.
Another useful type is what we call the \emph{spine} of a tetrahedron, consisting only of vertexes and edges which form a sequence:
\begin{definition}[spines]
 For a given semisimplicial type, the types of \emph{spines} are defined as:
 \begin{equation}
  \begin{alignedat}{4}
 & \spine 0 &\; & \defeq &\; \; & A_0 && \text{(i.e.\ just $\fulltetra 0$, which in turn is $A_0$)} \\
 & \spine 1 && \defeq && \Sigma (x_0, x_1 : A_0), (x_{01} : A_1(x_0,x_1)) && \text{(which is $\fulltetra 1$)} \\
 & \spine 2 && \defeq && \Sigma (x_i : A_0)_{\indexlist{2}{i}}, (x_{i(i+1)} : A_1(x_i, x_{i+1}))_{\indexlist{1}{i}} &\qquad & \text{(which is $\Lambda^2_1$)} \\
 & \spine 3 && \defeq && \Sigma (x_i : A_0)_{\indexlist{3}{i}}, (x_{i(i+1)} : A_1(x_i, x_{i+1}))_{\indexlist{2}{i}} && \text{(four points, three edges)} \\
 & \spine 4 && \defeq && \Sigma (x_i : A_0)_{\indexlist{4}{i}}, (x_{i(i+1)} : A_1(x_i, x_{i+1}))_{\indexlist{3}{i}} && \text{(five points, four edges)}\\
  \end{alignedat}
 \end{equation}
\end{definition}

There is a canonical projection $\phi_n : \fulltetra n \to \spine n$ which simply discards all triangle fillers and higher cells, and all edges apart from those in some $A_1(x_i, x_{i+1})$.
This also works if we replace the $\fulltetra n$ in the domain one of the other types that we have considered so far, since (at least for $n \geq 3$) all of these have strictly more components than $\spine n$.
In particular, we have $\phi_3 : \Lambda^3_1 \to \spine 3$ and $\phi_4 : \Lambda^4_1 \to \spine 4$.
From now, we assume that our semisimplicial type is in fact a semi-Segal type.

\begin{lemma}\label{lem:phi1}
For any $3$-restricted semi-Segal type, the canonical map
\begin{equation}
  \phi_3 : \Lambda^3_1 \to \spine 3
\end{equation}
is an equivalence.
\end{lemma}
\begin{proof}
The type $\Lambda^3_1$ is defined as a $\Sigma$-type with a number of components.  Note that, thanks to the naming convention that we are using, the type of any component of $\Lambda^3_1$ can be determined from the name of the corresponding variable, so we will refer to components simply by name in the following.

By reordering its components, we can see that $\Lambda^3_1$ is composed of:
\begin{itemize}
\item points $x_0, x_1, x_2, x_3$;
\item lines $x_{01}, x_{12}, x_{23}$;
\item a $\Lambda^2_1$-horn filler $x_{02}, x_{012}$;   \eqnum\label{eq:decomposing-Lambda31}
\item a $\Lambda^2_1$-horn filler $x_{13}, x_{123}$;
\item a $\Lambda^2_1$-horn filler $x_{03}, x_{013}$.
\end{itemize}

The Segal condition at level $2$ implies that the last three items in the list form contractible types, hence $\Lambda^3_1$ is equivalent to the $\Sigma$-type consisting of the first two items above, which is exactly $\spine 3$.
\end{proof}

We will also need a version of this lemma one level up.

\begin{lemma}\label{lem:phi2}
For any $4$-restricted semi-Segal type, the canonical map
\begin{equation}
  \phi_4 : \Lambda^4_1 \to \spine 4
\end{equation}
is an equivalence.
\end{lemma}
\begin{proof}
The proof proceeds in a similar fashion to the one for Lemma~\ref{lem:phi1}: we can reorder the components of $\Lambda^4_1$ with the ones of $\spine 4$ coming first, followed by a number of inner horn fillers, which are contractible by the Segal condition.  It is not hard to see that such a decomposition is possible (indeed, there are several ways of constructing one).
Nevertheless, we will present one explicitly, because it will be important for the proof of Lemma~\ref{lem:2-semicategory-semisegal} below.

The general strategy that we follow here can be visualised by starting with a horn $\Lambda^4_1$ and removing inner horn fillers one at a time until all that is left is the spine.  The whole process can be divided into three stages:
\begin{description}
  \item[stage 1] consider all the components of $\Lambda^4_1$ that contain $1$ as an internal vertex, and pair each of them with the corresponding component obtained by removing the vertex $1$;  order the resulting list of pairs by decreasing dimension, and remove them;
  \item[stage 2] remove the tetrahedron $x_{1234}$ together with its inner face $x_{134}$;
  \item[stage 3] use the decomposition of Lemma~\ref{lem:phi1} to remove all the remaining components of the tetrahedron $x_{1234}$, leaving its spine.
\end{description}

By carefully applying this procedure, and listing all the components in reverse order of removal, we obtain the following decomposition:
\begin{itemize}
\item elements of the spine: $x_0, x_1, x_2, x_3, x_4, x_{01}, x_{12}, x_{23}, x_{34}$;
\item horns of stage $3$: $(x_{234}, x_{24})$, $(x_{123}, x_{13})$, $(x_{124}, x_{14})$;
\item horn of stage $2$: $(x_{1234}, x_{134})$.   \eqnum\label{eq:decomposing-Lambda41}
\item $2$-horns of stage $1$: $(x_{012}, x_{02})$, $(x_{013}, x_{03})$, $(x_{014}, x_{04})$;
\item $3$-horns of stage $1$: $(x_{0123}, x_{023})$, $(x_{0134}, x_{034})$, $(x_{0124}, x_{024})$;
\end{itemize}

All the listed components except for the first line are grouped in such a way that they form contractible types of horn fillers, hence we get that the full type of horns $\Lambda^4_1$ is equivalent to the type determined by the first item in the list, i.e.\ $\spine 4$, as required.
\end{proof}

\begin{remark} \label{rem:on-semiSegal-definition}
 Recall that a \emph{simplicial set}, as used for example in categorical homotopy theory,
 where all inner horns can be filled (not necessarily uniquely)
 is called a \emph{weak Kan complex}~\citep{bv:htpy-invar}, or a \emph{quasi-category}~\citep{joyal2002quasi},
 or an $\infty$-category~\citep{lurie:higher-topoi}. 
 In Definition~\ref{def:semisegal-1-to-4}, we only require contractible horn filling for horns of the form $\Lambda^p_1$.
 This may at first sight seem too minimalistic.
 Further, in the study of \emph{Segal spaces}, it is often part of the definition that all maps $\phi_p : \fulltetra{p} \to \spine{p}$ are equivalences.
 In our setting one can show:
 
 $(*)$ For any semi-Segal type, the following are equivalent:
 \begin{enumerate}
  \item For all $p \geq 2$, all $\Lambda^p_1$-horns have contractible filling (as in Definition~\ref{def:semisegal-1-to-4}).
  \item All inner horns have contractible filling.
  \item All \emph{generalised inner horns} (in the sense of~\cite[Sec.~2.1.1]{joyal2008theory}) have contractible filling.
  \item For all $p$, the canonical projection $\phi_p : \fulltetra{p} \to \spine{p}$ is an equivalence.
 \end{enumerate}

 For example, if we know that horns in $\Lambda^2_1$ and $\Lambda^3_1$ have contractible filling,
 we can conclude that $\Lambda^3_2$-horns have contractible filling as well.
 To see this, consider $x : \Lambda^3_2$.
 Its type of fillers is the type of pairs $P \defeq (x_{013}, x_{0123})$.
 Note that the type of pairs $(x_{03}, x_{023})$ is a filler for a $\Lambda^2_1$-horn and thus contractible by assumption.
 Therefore, $P$ is equivalent to the type of tuples $(x_{03}, x_{023}, x_{013}, x_{0123})$.
 But $(x_{03},x_{013})$ is the filler of another $\Lambda^2_1$-horn.
 Hence, $P$ is in fact equivalent to the type of pairs $(x_{023}, x_{0123})$, which fills a $\Lambda^3_1$-horn and is thus contractible.
 We do not need the statement $(*)$ and thus omit a more general proof, which can be obtained using only arguments analogous to the ones we have shown so far.
 
 While we only consider the restricted case of $n \leq 4$, all these statements can be proved for any externally fixed number $n$ with the strategies that we have shown. 
 Unfortunately we cannot express this in Agda due to the limitation mentioned in Remark~\ref{rem:hott-is-too-weak}.
\end{remark}

\subsection{Equivalence of the Structures} \label{sec:proof-of-equivalence-sst-wildcats}

We recall Theorem~\ref{thm:wild-semi}:

\composition*

The proof of Theorem~\ref{thm:wild-semi} will be split into several steps (Lemmata~\ref{lem:0-semicategory-semisegal},~\ref{lem:1-semicategory-semisegal},~\ref{lem:2-semicategory-semisegal}), one for each of its points, except the first, which is trivial, since the definitions of graph and of 1-restricted semi-Segal type coincide.
We begin with a simple lemma, which will be used multiple times in the following.

\begin{lemma} \label{lem:unique-family}
 Let $X$ be a type.
 Then, having an element of $X$ is equivalent to having a type family over $X$ which is inhabited for ``exactly one $x: X$'', i.e.\ there is an equivalence 
 \begin{equation} \label{eq:unique-family}
  X \simeq \sm{F : X \to \UU} \iscontr(\sm{x:X} F(x)).
 \end{equation}
\end{lemma}
\begin{proof}
 Having a family $F : X \to \UU$ is equivalent to having $Y : \UU$ together with $f : Y \to X$ by~\cite[Thm 4.8.3]{hott-book}, and under this equivalence, the contractibility condition becomes $Y \simeq \mathbf{1}$. Thus, the right-hand side of~\eqref{eq:unique-family} is equivalent to $(\mathbf{1} \to X)$, and thus to $X$.
 
 Note that the function part of the such constructed equivalence~\eqref{eq:unique-family} is given by mapping $x$ to the family $F(y) \defeq (y = x)$ which has the required property thanks to contractibility of singletons.
\end{proof}

\begin{lemma}\label{lem:0-semicategory-semisegal}
The type of transitive graphs is equivalent to the type of 2-restricted semi-Segal types.
\end{lemma}
\begin{proof}
Given a 1-restricted semi-Segal type $A \jdeq (A_0, A_1)$, we need to show that the data needed to extend $A$ to a 2-restricted semi-Segal type, i.e.\ the type of the pair $(A_2, h_2)$ in Definition~\ref{def:semisegal-1-to-4}, is equivalent to the type of the composition operator $(\Comp)$ in Definition~\ref{def:graphs-1-to-4}.

Recall from~\eqref{eq:An-as-family-over-boundary} that the type of $A_2$ can be written as $\boundary^2 \to \UU$.
Since $\boundary^2$ is equivalent to 
\begin{equation}
  \Sigma (u : \Lambda^2_1), A_1(u_0, u_2),
\end{equation}
we can curry to see that the type of $A_2$ is equivalent to 
\begin{equation}
  \Pi(u : \Lambda^2_1), A_1(u_0, u_1) \to \UU.
\end{equation}
Thus, the type of the pair $(A_2, h_2)$ is equivalent to
\begin{equation}
  \Pi (u : \Lambda^2_1),
     \Sigma (\overline{A_2} : A_1(u_0, u_2) \to \UU),
     \iscontr \left(\Sigma_{A_1(u_0, u_2)} \overline{A_2} \right).
\end{equation}
By Lemma~\ref{lem:unique-family}, this is equivalent to
\begin{equation}
  \Pi (u : \Lambda^2_1), A_1(u_0, u_2),
\end{equation}
which is nothing else than a reformulation of the type of the composition operator.
\end{proof}

\begin{corollary}\label{cor:composition}
Let $A \equiv (A_0, A_1, A_2)$ be the $2$-restricted semi-Segal type corresponding to a transitive graph as in Lemma~\ref{lem:0-semicategory-semisegal}.  Then, for all objects $x_0$, $x_1$, $x_2$ and morphisms $x_{01} : \Hom(x_0, x_1), x_{12} : \Hom(x_1, x_2), x_{02} : \Hom(x_0, x_2)$,
on the left-hand side seen as points and edges, we have
\begin{equation} \label{eq:equivalence-A2-comp}
  A_2(\bound[x]{012}) \simeq (x_{12} \circ x_{01} = x_{02}).
\end{equation}
If the construction is performed in the canonical way, then the equivalence~\eqref{eq:equivalence-A2-comp} holds judgmentally,
\begin{equation} \label{eq:jdgm-A2-comp}
  A_2(\bound[x]{012}) \equiv (x_{12} \circ x_{01} = x_{02}).
\end{equation}
Furthermore, the unique horn filler for the $\Lambda^2_1$-horn determined by $x_{01}$ and $x_{12}$ is given by $(x_{12} \circ x_{01}, \refl_{x_{12} \circ x_{01}})$.
\end{corollary}
\begin{proof}
The equivalence~\eqref{eq:equivalence-A2-comp} is immediate from the construction in Lemma~\ref{lem:0-semicategory-semisegal}.
We can check each step to convince ourselves of the judgmental equality~\eqref{eq:jdgm-A2-comp}, which indeed does hold in our Agda formalisation.
\end{proof}

\begin{lemma}\label{lem:1-semicategory-semisegal}
The type of wild semicategories is equivalent to the type of $3$-restricted semi-Segal types.
\end{lemma}
\begin{proof}
We know that a $3$-restricted semi-Segal type is given by a $2$-restricted semi-Segal type, plus the type $A_3$ of fillers for $3$-dimensional boundaries (\refeq{eq:3d-tetra}), and the statement that $\Lambda^3_1$-horns have contractible filling, i.e.
  $h_3 : \Pi (u : \Lambda^3_1), \iscontr \left( \hornfillers u 3 1 \right)$.

First of all, it is easy to see, just by expanding the definitions, that the domain of $A_3 : \boundary^3 \to \UU$ is equivalent to the type
  $\Sigma (u : \Lambda^3_1), A_2(\bound[u]{023})$.

Therefore, just like in the proof of Lemma~\ref{lem:0-semicategory-semisegal}, we can curry both $A_3$ and $h_3$, and rewrite them into a single function of type
  $\Pi (u : \Lambda^3_1), T(u)$,
where
\begin{equation}
  T(u) \defeq
  \Sigma \left(\overline{A_3} : A_2(\bound[u]{023}) \to \UU\right),
    \iscontr \left(\Sigma \overline{A_3} \right),
\end{equation}
and, by Lemma~\ref{lem:unique-family}, we have that
\begin{equation}\label{eq:raw-associator-horn}
  T(u) \cong A_2(\bound[u]{023}).
\end{equation}
Using the isomorphism $\phi_3$ of Lemma~\ref{lem:phi1}, we get that this type can be further rewritten as
\begin{equation}\label{eq:raw-associator}
  \Pi (s : \spine 3), A_2 \left( \bound[\left(\phi_3^{-1}(s)\right)]{023} \right),
\end{equation}
so all that remains to be shown is that the type~\eqref{eq:raw-associator} above is equivalent to the type of the associator.

In order to understand what the type~\eqref{eq:raw-associator} looks like, we need to examine the construction of the isomorphism $\phi_3$ of Lemma~\ref{lem:phi1}, and, more specifically, its inverse $\phi_3^{-1} : \spine 3 \to \Lambda^3_1$.  Recall that, in the proof of Lemma~\ref{lem:phi1}, all components of $\Lambda^3_1$ that are not present in $\spine 3$ were grouped into pairs consisting of the elements of an inner horn filler.  
Thus, we can assume that $A_2$ as well as the components of the map $\phi_3^{-1}$ have the form described in Corollary~\ref{cor:composition}.

So, let $s : \spine 3$, and set $y \defeq \phi_3^{-1}(s)$.  By going through the decomposition of $\Lambda^3_1$ in terms of contractible horn filling as given in~\eqref{eq:decomposing-Lambda31}, we get:
\begin{equation}
\begin{alignedat}{6}
  & y_{02} \equiv y_{12} \circ y_{01},\; && y_{012} \equiv \refl_{};   &\qquad
  & y_{13} \equiv y_{23} \circ y_{12},\; && y_{123} \equiv \refl_{};         &\qquad
  & y_{03} \equiv y_{13} \circ y_{01},\; && y_{013} \equiv \refl_{}.
\end{alignedat}
\end{equation}
We can now calculate the type $A_2(\bound[y]{023})$ as follows:
\begin{equation}
\begin{alignedat}{2}
  &&& A_2(\bound[y]{023}) \\
  & \cong & \quad & (y_{23} \circ y_{02} = y_{03}) \\
  & \cong &   & (y_{23} \circ (y_{12} \circ y_{01}) = (y_{23} \circ y_{12}) \circ y_{01}),
\end{alignedat}
\end{equation}
and this last expression exactly matches the type of the associator given in Definition~\ref{def:graphs-1-to-4}.
\end{proof}

\begin{lemma}\label{lem:1-general-assoc}
Let $A \jdeq (A_0, A_1, A_2, A_3)$ the $3$-restricted semi-Segal type corresponding to a wild semicategory as in Lemma~\ref{lem:1-semicategory-semisegal}.  Let $u : \Lambda^4_1$ be a horn in $A$.  Define the \emph{generalised associator} of $u$ by:

\begin{equation}
\begin{alignedat}{2}
  & \widehat\alpha : \ \Pi (u : \Lambda^4_1), A_2(\bound[u]{023}) \\
  & \widehat\alpha(u) \defeq
    \ \mapfunc {u_{23} \circ \_} (u_{012}^{-1}) \ct
     \alpha(u_{01}, u_{12}, u_{23}) \ct
     \mapfunc {\_ \circ u_{01}} (u_{123}) \ct
     u_{013}.
\end{alignedat}
\end{equation}
Then, if $(u, f)$ is the boundary of a tetrahedron in $A$, with $u$ being the corresponding $\Lambda^4_1$-horn and $f$ the remaining face, we have:
\begin{equation}\label{eq:gen-associator-iso}
  A_3(u, f) \cong (\widehat\alpha(u) = f),
\end{equation}
and furthermore, the unique filler for $u$ is equal to $(\widehat\alpha(u), \refl_{\widehat\alpha(u)})$.
\end{lemma}
\begin{proof}
By examining the proof of Lemma~\ref{lem:1-semicategory-semisegal}, we see that for all $s : \spine 3$, and $f : A_2(\bound[u]{023})$, where $u \defeq \phi_3^{-1}(s)$,
\begin{equation}
  A_3(u, f) \cong (\alpha(s) = f),
\end{equation}
so the first assertion follows immediately from the observation that the type $\alpha(s) = f$ is equivalent to $\widehat\alpha(\phi_3^{-1}(s)) = f$ and the fact that $\phi$ is an equivalence.

The second assertion is an immediate consequence of the first, since the type of horn fillers is manifestly equivalent to the type of singletons of $\widehat\alpha(u)$.
\end{proof}

In the following, we will assume that the isomorphism of Lemma~\ref{lem:1-semicategory-semisegal} maps any wild semicategory into a 3-restricted semi-Segal type $A \jdeq (A_0, A_1, A_2, A_3)$ where
\begin{equation}
  A_3(u, f) \equiv (\widehat \alpha(u) = f).
\end{equation}

This is possible thanks to Lemma~\ref{lem:1-general-assoc}.

\begin{lemma}\label{lem:2-semicategory-semisegal}
The type of wild $2$-semicategories is equivalent to the type of $4$-restricted semi-Segal types.
\end{lemma}
\begin{proof}
We will prove this in two steps, following the same strategy as the proof of Lemma~\ref{lem:1-semicategory-semisegal}.
First, we show that the extra data needed to turn a $3$-restricted semi-Segal type into a $4$-restricted semi-Segal type is equivalent to the type
\begin{equation} \label{eq:raw-pentagon}
  \Pi (s : \spine 4), A_3\left(\bound[{{\left(\phi_4^{-1}(s)\right)}}]{0234}\right).
\end{equation}
This can be proved in exactly the same way as the corresponding statement in the proof of Lemma~\ref{lem:1-semicategory-semisegal}.

The next step, however, requires some new calculation.  We need to show that the type~\eqref{eq:raw-pentagon} is equivalent to that of the pentagon $\Penta$ in~\eqref{eq:type-of-penta}.

So, let $s : \spine 4$, and again set $y \defeq \phi_4^{-1}(s)$.  Just like in the proof of Lemma~\ref{lem:1-semicategory-semisegal}, we can go through the list of horn fillers contained in the decomposition of $\Lambda^4_1$ as listed in~\eqref{eq:decomposing-Lambda41}, and compute the following values. This time, we omit all the components which are simply $\refl$.  We will write $y|ijkl$ for indices $0 \leq i < j < k < l \leq 4$ to denote the $\Lambda^3_1$-horn $(\horntuple[y]{ijkl}{j})$.
\begin{equation}
\begin{aligned}
  & y_{24} \equiv y_{34} \circ y_{23},
  \quad y_{13} \equiv y_{23} \circ y_{12},
  \quad y_{14} \equiv y_{24} \circ y_{12}, \\
  & y_{134} \equiv \widehat\alpha(y|1234) = \alpha(y_{12}, y_{23}, y_{34}), \\
  & y_{02} \equiv y_{12} \circ y_{01},
  \quad y_{03} \equiv y_{13} \circ y_{01},
  \quad y_{04} \equiv y_{14} \circ y_{01}, \\
  & y_{024} \equiv \widehat\alpha(y|0124) = \alpha(y_{01}, y_{12}, y_{24}), \\
  & y_{034} \equiv \widehat\alpha(y|0134) = \alpha(y_{01}, y_{13}, y_{34}) \ct
                                           \mapfunc {\_ \circ y_{01}} (y_{134}) \\
  & y_{023} \equiv \widehat\alpha(y|0123) = \alpha(y_{01}, y_{12}, y_{23}),
\end{aligned}
\end{equation}
The calculation now proceeds as follows:
\begin{equation}
\begin{alignedat}{2}
   && \quad & A_3(\bound[y]{0234}) \\
   & \cong && \left( y_{034} = \widehat\alpha(y|0234) \right) \\
   & \cong && \left( y_{034} =
        \mapfunc {y_{34} \circ \_} (y_{023}^{-1}) \ct
        \alpha(y_{02}, y_{23}, y_{34}) \ct
        \mapfunc {\_ \circ y_{02}} (y_{234}) \ct
        y_{024} \right) \\
    & \cong && \Big(\alpha(y_{01}, y_{13}, y_{34}) \ct \mapfunc {\_ \circ y_{01}} (y_{134}) = \\
    &&& \qquad
        \mapfunc {y_{34} \circ \_} (\alpha(y_{01}, y_{12}, y_{23})^{-1}) \ct
        \alpha(y_{02}, y_{23}, y_{34}) \ct
        \alpha(y_{01}, y_{12}, y_{24}) \Big) \\
    & \cong && \Big( \alpha(y_{01}, y_{12}, y_{23} \circ y_{34}) \ct
    \alpha(y_{12} \circ y_{01}, y_{23}, y_{34})
    = \\
    &&& \qquad \mapfunc {\_ \circ y_{01}} (\alpha(y_{12}, y_{23}, y_{34})) \ct
    \alpha(y_{01}, y_{23} \circ y_{12}, y_{34}) \ct
    \mapfunc {y_{34} \circ \_} (\alpha(y_{01}, y_{12}, y_{23})) \Big). 
\end{alignedat} 
\end{equation}

\vspace{-\the\baselineskip}
\vspace{-0.3cm}
\end{proof}

\vspace{0.1cm}

\section{Identity and Degeneracy Structure} \label{sec:identity-and-degeneracy}

After the detailed discussion on composition structures in the previous section, we are ready to talk about identities.
As before, we consider both the ``standard'' presentation and an encoding via semi-Segal types.

\subsection{Identities for Wild Semicategories}

\begin{definition}[identities for wild structures] \label{def:id-structure}
 We equip the structures of Definition~\ref{def:graphs-1-to-4} with identities as follows:
 \begin{enumerate}
  \item A \emph{reflexive-transitive graph} is a transitive graph $(\Ob, \Hom, \_ \circ \_)$ together with a family $\Ids : \Pi x, \Hom(x,x)$.
  \item A \emph{wild precategory} is a wild semicategory $(\Ob, \Hom, \_ \circ \_, \alpha)$, equipped with $\Ids$ as above and equalities $\lambda : \Pi xy, (f : \Hom(x,y)), \Ids_y \circ f = f$ as well as $\rho : \Pi xy, (f : \Hom(x,y)), f \circ \Ids_x = f$.
  \item A \emph{wild 2-precategory} is a wild 2-semicategory $(\Ob, \Hom, \Comp, \alpha, \Penta)$ together with $\Ids$, $\lambda$, $\rho$ as above, an equality 
  \begin{equation} \label{eq:id-triangle-1}
   t_1: \; \Pi xyz, (f : \Hom(x,y)), (g : \Hom(y,z)), \; 
   \mapfunc {g \circ \_} (\lambda_f)
   =
   \alpha_{f,\Ids,g} \ct \mapfunc {\_ \circ f} (\rho_g)
  \end{equation}
  which can be pictured as:
  \begin{equation}  \label{eq:t1-diagram-picture}
  \begin{tikzpicture}[x=2.2cm,y=-1.0cm, baseline=(current bounding box.center)]
   \node (T1) at (0,0) {$g \circ (\Ids \circ f)$};
   \node (T2) at (2,0) {$(g \circ \Ids) \circ f$};
   \node (T3) at (1,1) {$g \circ f$};
   
   \draw[->] (T1) to node [above] {\scriptsize $\alpha_{f,\Ids,g}$} (T2);
   \draw[->] (T1) to node [below left] {\scriptsize $\mapfunc {(g \circ \_)} (\lambda_f)$} (T3);
   \draw[->] (T2) to node [below right] {\scriptsize $\mapfunc {\_ \circ f} (\rho_g)$} (T3);
  \end{tikzpicture}
  \end{equation}
  and two equalities
  \begin{align}
   &t_0 : \;   \Pi xyz, (f : \Hom(x,y)), (g : \Hom(y,z)), \; 
   \lambda_{g \circ f}
   =
   \alpha_{f,g,\Ids} \ct \mapfunc {\_ \circ f} (\lambda_g)
   \label{eq:id-triangle-0}\\
   &t_2 : \;   \Pi xyz, (f : \Hom(x,y)), (g : \Hom(y,z)), \; 
   \mapfunc {g \circ \_} (\rho_f)
   = 
   \alpha_{\Ids,f,g} \ct \rho_{g \circ f}
   \label{eq:id-triangle-2} 
  \end{align}
  which can be represented as follows:
  \begin{equation}
  \begin{tikzpicture}[x=2cm,y=-1.0cm, baseline=(current bounding box.center)]
   \node (T1) at (0,0) {$\Ids \circ (g \circ f)$};
   \node (T2) at (2,0) {$(\Ids \circ g) \circ f$};
   \node (T3) at (1,1) {$g \circ f$};
   
   \draw[->] (T1) to node [above] {\scriptsize $\alpha_{f,g,\Ids}$} (T2);
   \draw[->] (T1) to node [below left] {\scriptsize $\lambda_{g \circ f}$} (T3);
   \draw[->] (T2) to node [below right] {\scriptsize $\mapfunc {\_ \circ f} (\lambda_g)$} (T3);

   \node (T1B) at (3.2,0) {$g \circ (f \circ \Ids)$};
   \node (T2B) at (5.2,0) {$(g \circ f) \circ \Ids$};
   \node (T3B) at (4.2,1) {$g \circ f$};
   
   \draw[->] (T1B) to node [above] {\scriptsize $\alpha_{\Ids,f,g}$} (T2B);
   \draw[->] (T1B) to node [below left] {\scriptsize $\mapfunc {g \circ \_} (\rho_f)$} (T3B);
   \draw[->] (T2B) to node [below right] {\scriptsize $\rho_{g \circ f}$} (T3B);
  \end{tikzpicture}
  \end{equation}
 \end{enumerate}
\end{definition}

\begin{remark} \label{rem:two-triangles-too-much}
 One could argue that the last two components (\ref{eq:id-triangle-0},~\ref{eq:id-triangle-2}) should not be part of the definition of a wild 2-precategory.
 If one checks the definition of a (weak) 2-category in a textbook on category theory, one will most likely \emph{not} encounter these two triangles.
 The reason is that they can be derived from the other data (which includes the ``middle'' triangle~\eqref{eq:id-triangle-1}) and thus do not need to be listed separately.
 For us, the situation is different since we are not stating \emph{laws} but \emph{structure}: although the two triangles in question can be derived in our setting as well (see Lemma~\ref{lem:id-triangles-01-derivable}), the \emph{type} of wild 2-precategories with those triangles is not necessarily equivalent to the one without them.
 
 Our recipe for determining the components of wild $n$-semicategories and precategories (for small $n$) is to start writing down the (infinite) composition and identity structure, but simply stop and ``cut off'' everything above the corresponding level.
 If we look at the definition of a tricategory~\cite[p.~25f]{nick:tricats}, the two triangles~(\ref{eq:id-triangle-0},~\ref{eq:id-triangle-2}) are part of the definition, and they are \emph{at the same level as~\eqref{eq:id-triangle-1}}; hence, it feels correct to us to include them here.
 
 Having given this argument, it does ultimately not matter whether we include the triangles.
 When we pass from \emph{wild} structure (which in the end is mostly an auxiliary concept) to well-behaved structure, we will ask for a truncation condition which ensures that~(\ref{eq:id-triangle-0},~\ref{eq:id-triangle-2}) are propositions.
 Thus, for $2$-categories, it will be the case that the type with those triangles is equivalent to the type without those triangles, 
 i.e.\ we could omit them for the same reason as they are omitted in set-based presentations of 2-categories.
\end{remark}

We further want to mention that we get a coherence condition between $\alpha$ and $\lambda$, $\rho$ automatically,
just as it is the case in standard 2-category theory. 
In a nutshell, $\alpha$ commutes with $\lambda$ or $\rho$.
In type theory, this is extremely simple:
\begin{auxlemma*}
 Given composable morphisms $f$, $g$, $h$,
 the following square commutes up to homotopy:
 \begin{equation}
\begin{tikzpicture}[x=4cm,y=-2.7cm, baseline=(current bounding box.center)]
\node (UL) at (0,0) {$(h \circ (\Ids \circ g)) \circ f$}; 
\node (UR) at (1,0) {$h \circ ((\Ids \circ g) \circ f)$}; 
\node (LL) at (0,1) {$(h \circ g) \circ f$}; 
\node (LR) at (1,1) {$h \circ (g \circ f)$}; 

\draw[->] (UL) to node [above] {\scriptsize $\alpha_{f, (\Ids \circ g), h}$} (UR);
\draw[->] (LL) to node [below] {\scriptsize $\alpha_{f, g, h}$} (LR);

\draw[->] (UL) to node [left] {\scriptsize $\mapfunc {\_ \circ f} (\mapfunc {h \circ \_} (\lambda_g))$} (LL);
\draw[->] (UR) to node [right] {\scriptsize $\mapfunc {h \circ \_} (\mapfunc {\_ \circ f} (\lambda_g))$} (LR);
\end{tikzpicture}
\end{equation}
Analogously, the diagram we get by swapping $g$ and $\Ids$ and using $\rho_g$ instead of $\lambda_g$ commutes, 
as well as the similar squares that pair $\Ids$ with $f$ or $h$ instead of $g$.
\end{auxlemma*}
\begin{proof}
 We simply formulate the statement with $\Ids \circ g$ replaced by a generic morphism $g'$, and with $\lambda_g$ replaced by any proof $p : g' = g$, and do path induction on $p$.
\end{proof}

\begin{lemma} \label{lem:id-triangles-01-derivable}
 Given a wild 2-precategory without the two components~(\ref{eq:id-triangle-0},~\ref{eq:id-triangle-2}), 
 elements of these components are derivable.
\end{lemma}
\begin{proof}
Consider the following diagram:

\newdimen\R
\newdimen\S
\R=0.4\textwidth
\S=0.08\textwidth

\newcommand{\rot}{-54}

\begin{equation} \label{eq:3-coherator}
\begin{tikzpicture}[baseline=(current bounding box.center)]
\node (P1) at (\rot-0:\R) {\fbox{$((h \circ \Ids) \circ g) \circ f$}}; 
\node (P2) at (\rot-72:\R) {\fbox{$(h \circ (\Ids \circ g)) \circ f$}}; 
\node (P3) at (\rot-144:\R) {\fbox{$h \circ ((\Ids \circ g) \circ f)$}}; 
\node (P4) at (\rot-216:\R) {\fbox{$h \circ (\Ids \circ (g \circ f))$}}; 
\node (P5) at (\rot-288:\R) {\fbox{$(h \circ \Ids) \circ (g \circ f)$}}; 

\node (P6) at (\rot-36:\S) {\fbox{$(h \circ g) \circ f$}};
\node (P7) at (\rot-216:\S) {\fbox{$h \circ (g \circ f)$}};

\draw[<-, tikzshortarrow] (P1) to node [below] {\scriptsize $\mapfunc {\_ \circ f} (\alpha(g, \Ids, h))$} (P2);
\draw[<-, tikzshortarrow] (P2) to node [left] {\scriptsize $\alpha(f, \Ids \circ g, h)$} (P3);
\draw[<-, tikzshortarrow] (P1) to node [right] {\scriptsize $\alpha(f, g, h \circ \Ids)$} (P5);
\draw[<-, tikzshortarrow] (P5) to node [above right] {\scriptsize $\alpha(g \circ f, \Ids, h)$} (P4);
\draw[->, tikzshortarrow] (P4) to node [above left] {\scriptsize $\mapfunc {h \circ \_} (\alpha(f,g,\Ids))$} (P3);

\draw[<-, tikzshortarrow] (P6) to node [right] {\scriptsize $\alpha(f,g,h)$} (P7);

\draw[->, tikzshortarrow] (P1) to node [right, near end] {\scriptsize $\mapfunc {\_ \circ f} (\mapfunc {\_ \circ g} (\rho_h))$} (P6);
\draw[->, tikzshortarrow] (P2) to node [left, near end] {\scriptsize $\mapfunc {\_ \circ f} (\mapfunc {h \circ \_} (\lambda_g))$} (P6);
\draw[->, tikzshortarrow] (P3) to node [below] {\scriptsize $\mapfunc {h \circ \_} (\mapfunc {\_ \circ f} (\lambda_g))$} (P7);
\draw[->, tikzshortarrow] (P4) to node [right] {\scriptsize $\mapfunc {h \circ \_} (\lambda_{g \circ f}))$} (P7);
\draw[->, tikzshortarrow] (P5) to node [above] {\scriptsize $\mapfunc {\_ \circ (g \circ f)} (\rho_h)$} (P7);
\end{tikzpicture}
\end{equation}
In this diagram, the two quadrangles commute up to homotopy by the auxiliary lemma above.
The top right diagram commutes due to~\eqref{eq:id-triangle-1}, applying $\mapfunc {h \circ \_}$ on the proof and using functoriality of $\mapfunc{}$~\cite[Lemma~2.2.2(i)]{hott-book}.
In the same way, the bottom triangle commutes, and so does the outermost pentagon thanks to $\Penta$.
Thus, we get commutativity of the top left triangle.
This triangle \emph{nearly} shows the equality~\eqref{eq:id-triangle-0}, but not quite, because $\mapfunc {h \circ \_}$ is applied everywhere.
Let us choose $h \defeq \Ids$.
The function $(\Ids \circ \_)$ is equal to the identity, hence an equivalence.
Therefore, $\mapfunc {\Ids \circ \_}$ is an equivalence as well, which proves~\eqref{eq:id-triangle-0}.
The argument for~\eqref{eq:id-triangle-2} is completely analogous.
\end{proof}

\begin{remark}
 If we were to define a 3-categorical structure, the diagram~\eqref{eq:3-coherator} would be precisely one of the coherators which we would need to connect $t_0$ and $t_1$.
 A version of it can found among the axioms of a 3-category in~\cite[p.~26]{nick:tricats}.
\end{remark}

\subsection{Degeneracies in Semisimplicial Types}

Recall that a simplicial set
can be described as a family $(X_n)_{n \in \mathbb N}$ of sets, together with face maps $d^n_i :X_n \to X_{n-1}$ (where $n \geq 1$ and $0 \leq i \leq n$) and degeneracy maps $s^n_i : X_n \to X_{n+1}$ (where $n \geq 0$ and $0 \leq i \leq n$), 
such that the following so-called \emph{simplicial identities} are satisfied:
\begin{align}
  & d^{n+1}_i \circ d^n_j = d^{n+1}_{j-1} \circ d^n_i && \text{for } i < j   \label{eq:simpli-id-d-d} \\
  & d^n_i \circ s^{n-1}_j = s^n_{j-1} \circ d^{n+1}_i && \text{for } i < j \label{eq:simpli-id-s-d-1} \\
  & d^n_i \circ s^{n-1}_j = \mathsf{id} && \text{for } i = j \text{ or } i = j + 1 \label{eq:simpli-id-s-d-2} \\
  & d^n_i \circ s^{n-1}_j = s^n_j \circ d^{n+1}_{i-1} && \text{for } i > j + 1 \label{eq:simpli-id-s-d-3} \\
  & s^n_i \circ s^{n+1}_j = s^n_{j+1} \circ s^{n+1}_i && \text{for } i \leq j \label{eq:simpli-id-s-s}
\end{align}
In our setting, we have already discussed the face maps in Section~\ref{sec:semisimplicialtypes}.
Since the face maps are simply projections, the identity~\eqref{eq:simpli-id-d-d} turns out to hold judgmentally.
As our next step, we want to define the notion of a degeneracy structure for a given semisimplicial type $(A_0, \ldots, A_n)$.
We would like to make the simplicial identities hold judgmentally, but since we cannot express this condition as a type, it has to follow from a suitable encoding.
We cannot define a structure of degeneracy maps which makes \emph{all} the simplicial identities hold, but fortunately, this is not a problem:
the last equation~\eqref{eq:simpli-id-s-s} bears no importance for our further plans.

Assume we have defined $s^j_i$ for $j < n$.
Observe that the equations (\ref{eq:simpli-id-s-d-1},~\ref{eq:simpli-id-s-d-2},~\ref{eq:simpli-id-s-d-3}) 
already determine the complete boundary of the (n+1)-dimensional tetrahedron $s^n_i(x)$.
Thus, the strategy for ensuring that (\ref{eq:simpli-id-s-d-1},~\ref{eq:simpli-id-s-d-2},~\ref{eq:simpli-id-s-d-3}) hold is that we do not simply ask for 
functions $\fulltetra n \to \fulltetra {{n+1}}$, but instead for dependent functions that only choose a filler for the appropriate boundary.
We are only interested in the cases $n \in \{0,1,2\}$, for which this can be stated explicitly as follows:

\begin{definition}[degeneracy structure] \label{def:degeneracy-structure}
 Assume we have a semisimplicial type $(A_0, \ldots, A_n)$. In the following, we assume that $n$ is at least $1$ (or $2$ or $3$, respectively).
 \begin{enumerate}
  \item A \emph{1-degeneracy structure} on the semisimplicial type is simply a function
  \begin{alignat}{2}
   &s^0_0 : \mathrlap{(x_0 : A_0) \to A_1(x_0, x_0).}
  \intertext{\item A \emph{2-degeneracy structure} is $s_0^0$ as above together with
  functions $s^1_0$ and $s^1_1$ of the following types:}
   & s^1_0 : \mathrlap{((x_0, x_1, x_{01}) : \fulltetra{1}) \to A_2(\bound[u]{012})} \\
   & \qquad \textit{where} && u_0 \defeq u_1 \defeq x_0; \, u_2 \defeq x_1; \, u_{01} \defeq s^0_0(x_0); \, u_{02} \defeq u_{12} \defeq x_{01}  \notag \\
   & s^1_1 : \mathrlap{((x_0, x_1, x_{01}) : \fulltetra{1}) \to A_2(\bound[u]{012})} \\
   & \qquad \textit{where} && u_0 \defeq x_0; \, u_1 \defeq u_2 \defeq x_1; \, u_{01} \defeq u_{02} \defeq x_{01}; \, u_{12} \defeq s^0_0(x_1). \notag \\
  \intertext{\item A \emph{3-degeneracy structure} is $s_0^0$, $s^1_0$, $s^1_1$ as above together with functions $s^2_0$, $s^2_1$, $s^2_2$ as follows: }
   & s^2_0 : \mathrlap{(\tetratuple{012}) : \fulltetra{2}) \to A_3(\bound[u]{0123})}  \\
   & \qquad \textit{where} &\quad & u_0 \defeq u_1 \defeq x_0; \, u_2 \defeq x_1; \, u_3 \defeq x_2; \, \notag \\
      &&& u_{01} \defeq s^0_0(x_0); \, u_{02} \defeq u_{12} \defeq x_{01}; \, u_{03} \defeq u_{13} \defeq x_{02}; \, u_{23} \defeq x_{12}; \, \notag \\
      &&& u_{012} \defeq s^1_0(\tetratuple{01}); \, u_{013} \defeq s^1_0(\tetratuple{02}); \, u_{023} \defeq x_{012} \notag \\
   & s^2_1 : \mathrlap{(\tetratuple{012}) : \fulltetra{2}) \to A_3(\bound[u]{0123})} \\
   & \qquad \textit{where} &\quad & u_0 \defeq x_0; \, u_1 \defeq u_2 \defeq x_1; \, u_3 \defeq x_2; \, \notag \\
      &&& u_{01} \defeq u_{02} \defeq x_{01}; \, u_{12} \defeq s^0_0(x_1); \, u_{03} \defeq x_{02}; \, u_{13} \defeq u_{23} \defeq s^0_0(x_1); \, \notag \\
      &&& u_{012} \defeq s^1_1(\tetratuple{01}); \, u_{013} \defeq u_{023} \defeq x_{012}; \, u_{123} \defeq s^1_0(\tetratuple{12}) \notag \\
   & s^2_2 : \mathrlap{(\tetratuple{012}) : \fulltetra{2}) \to A_3(\bound[u]{0123})}  \\
   & \qquad \textit{where} &\quad & u_0 \defeq x_0; \, u_1 \defeq x_1; \, u_2 \defeq u_3 \defeq x_2; \, \notag \\
      &&& u_{01} \defeq x_{01}; \, u_{02} \defeq u_{03} \defeq x_{02}; \, u_{12} \defeq u_{13} \defeq x_{12}; \, u_{23} \defeq s^0_0(x_2); \, \notag \\
      &&& u_{012} \defeq u_{013} \defeq x_{012}; \, u_{023} \defeq s^1_1(\tetratuple{02}); \, u_{123} \defeq s^1_1(\tetratuple{12}).\notag 
  \end{alignat}
 \end{enumerate}
\end{definition}

\begin{definition}[semi-Segal type with degeneracies] \label{def:semi-segal-with-deg}
 For $n \in \{2,3,4\}$, we say that an \emph{$n$-restricted semi-Segal type with degeneracies} is an $n$-restricted semi-Segal type with $(n-1)$-degeneracy structure.
\end{definition}

\begin{remark} \label{rem:deg-structure-one-lower}
 It is intentional that the degeneracy structure is of a lower level than the semi-Segal strucure. For example, a $3$-restricted semi-Segal type with degeneracies has a family $A_3$ of which the degeneracy structure does not make use.
 Intuitively, this is because the highest piece of identity structure should only make use of the second-highest piece of composition structure: 
 In Definition~\ref{def:id-structure}, we see that the coherators of identity~(\ref{eq:id-triangle-0},~\ref{eq:id-triangle-1},~\ref{eq:id-triangle-2}) do make use of the associator, but not of its coherator (the pentagon).
 This will become clearer in Theorem~\ref{thm:identities-degeneracies} below.
\end{remark}

\subsection{Correspondence between Identities and Degeneracies}
\label{sec:id-goes-to-deg}

The following lemma makes use of the fact that an $n$-degeneracy structure already makes sense for an $n$-restricted semi-Segal type (as suggested in Remark~\ref{rem:deg-structure-one-lower}).

\begin{lemma} \label{lem:id-goes-to-deg}
 Under the equivalence
 constructed in Theorem~\ref{thm:wild-semi}, the identity structure on a
 reflexive-transitive graph (a wild semicategory, a wild 2-semicategory) gets
 mapped to a $1$ ($2$, $3$)-degeneracy structure of a $1$ ($2$, $3$)-restricted semi-Segal type.
\end{lemma}
Lemma~\ref{lem:id-goes-to-deg} immediately implies the following connection:
\begin{theorem} \label{thm:identities-degeneracies}
 The structures in Definition~\ref{def:id-structure} are equivalent to the structures in Definition~\ref{def:semi-segal-with-deg}:
 \begin{enumerate}
  \item The type of reflexive-transitive graphs is equivalent to the type of $2$-restricted semi-Segal types with degeneracies.
  \item The type of wild precategories is equivalent to the type of $3$-restricted semi-Segal types with degeneracies.
  \item The type of wild $2$-precategories is equivalent to the type of $4$-restricted semi-Segal types with degeneracies. \qed
 \end{enumerate}
\end{theorem}

\begin{proof}[Proof of Lemma~\ref{lem:id-goes-to-deg}]
  The statement is completely obvious for $1$-restricted semi-Segal types, since
  a $1$-degeneracy structure is exactly the same thing as an identity structure
  for a transitive graph.

  As for $2$-restricted semi-Segal types, it follows immediately from
  Corollary~\ref{cor:composition} that the type of the function $s^1_0$ which is
  part of a $2$-degeneracy structure is mapped to the type of a right unitor
  $\rho$ for the corresponding wild precategory, and similarly the type of
  $s^1_1$ is mapped to that of a left unitor.

  The case of $3$-restricted semi-Segal types is more involved, but it
  essentially amounts to a straightforward type calculation 
  (much of which is judgmental and thus automatic in our Agda formalisation). 
  We will only deal
  with the correspondence of the intermediate degeneracy $s^2_1$ with the main
  triangular coherence of a wild 2-precategory~\eqref{eq:id-triangle-1}, as the two other ones are analogous.

  If $X$ is a 3-restricted semi-Segal type, the type of 2-simplices $\fulltetra{2}$
  is equivalent to the type of spines $\spine 2$ thanks to the Segal condition.
  By rewriting along this equivalence, we get that the type of $s^2_1$ is
  equivalent to the type of functions that map a spine determined by morphisms
  $f : \Hom(x, y)$ and $g : \Hom(y, z)$ into the tetrahedron $x : \fulltetra 3$,
  where $x_{012} \jdeq \lambda_f$, $x_{123} \jdeq \lambda_g$, and the other faces are
  reflexivity proofs. Here we are implicitly using
  Corollary~\ref{cor:composition} to identify a 2-dimensional simplex with an
  equality of morphisms.

  Applying the isomorphism~\eqref{eq:gen-associator-iso}, we can rewrite the
  type of this tetrahedron into the equation
  \begin{equation}
    \mapfunc{g \circ \_} (\lambda_f^{-1}) \ct
    \alpha(f, \Ids_y, g) \ct
    \mapfunc {\_ \circ f} (\lambda_g) \ct
    \refl = \refl,
  \end{equation}
  which can be equivalently expressed as
  \begin{equation}
    \alpha(f, \Ids_y, g) \ct
    \mapfunc {\_ \circ f} (\lambda_g)
    = \mapfunc{g \circ \_} (\lambda_f).
  \end{equation}
\end{proof}

\begin{corollary}\label{cor:deg-derivable}
  Assume we have a 4-restricted semi-Segal type with degeneracies $s^0_0$,
  $s^1_0$, $s^1_1$, and $s^2_1$. We can derive degeneracies $s^2_0$ and $s^2_2$.
\end{corollary}
\begin{proof}
  Translating via Theorem~\ref{thm:identities-degeneracies}, this becomes the
  statement of Lemma~\ref{lem:id-triangles-01-derivable}.
\end{proof}

\subsection{Uniqueness of the Identity Structure}

Our usage of the attribute \emph{wild} in notions such as \emph{wild semicategory} indicates that higher levels of the structure are not ``controlled''; coherence is not guaranteed.
To give an example, the $\lambda$ and $\rho$ of a wild semicategory are not required to satisfy any further equality, as the corresponding rules are only added when considering a wild 2-semicategory.
The problems of wild structures are that they are not preserved under certain operations.
For example, given a wild semicategory and one of its objects, one can in general not perform a slice construction which produces a wild semicategory again.
This is essentially the same effect that is visible in the usual theory of bicategories, where the pentagon law is required to construct an associator for the slice bicategory.
In our type-theoretic setting, we can avoid wildness by requiring our structures to be truncated to ensure that all wanted equalities hold.
Thus, we define:

\begin{definition}[dropping wildness] \label{def:drop-wildness}
 We define the following ``non-wild'' structures:
 \begin{enumerate}
  \item A \emph{preordered set} is a reflexive-transitive graph where $\Hom(x,y)$ is a proposition for all objects $x,y$.
  \item A \emph{precategory} is a wild precategory where $\Hom(x,y)$ is always a set.
  \item A \emph{2-precategory} is a wild 2-precategory where $\Hom(x,y)$ is always a $1$-type.
 \end{enumerate}
\end{definition}

If we have one of the wild structures without identities from Definition~\ref{def:graphs-1-to-4} (e.g.\ a wild semicategory), 
it is not always be possible to equip it with an identity structure.
This is to be expected: for example, $\Hom$ could be the empty type everywhere.
Even if it is possible to find an identity structure, there is in general not a \emph{unique} way of doing it.
In other words, the type of identity structures is not a proposition.
However, it \emph{is} a proposition if the type of morphisms is truncated at an appropriate level.
This can be phrased as follows:
\begin{theorem} \label{thm:id-structure-unique}
 There is at most one way to extend one of the wild structures of Definition~\ref{def:graphs-1-to-4} (\ref{item:graphs-1-to-4-snd}-\ref{item:graphs-1-to-4-frth}) to the corresponding structure of Definition~\ref{def:drop-wildness}, in the following sense:
 \begin{enumerate} 
  \item Given a transitive graph, the structure needed to extend it to a preordered set inhabits a proposition. \label{item:is-prop-1}
  \item Similarly, for a wild semicategory, the structure which makes it a precategory inhabits a proposition. \label{item:is-prop-2}
  \item Given a wild 2-semicategory, the structure required for a 2-precategory inhabits a proposition. \label{item:is-prop-3}
 \end{enumerate}
\end{theorem}
\begin{proof}
 For~\eqref{item:is-prop-1}, observe that the additional structure needed is 
 \begin{equation}
  \left(\Pi x, \Hom(x,x)\right) \times \Pi xy, \isprop(\Hom(x,y)).
 \end{equation}
 The second factor is a proposition, and an inhabitant of it implies that the first factor is a proposition as well.
 
 Property~\eqref{item:is-prop-2} is not much harder. 
 The structure we need to add to a wild semicategory in order to obtain a wild precategory are $\Ids$, $\lambda$, $\rho$, and a proof that $\Hom$ is a family of sets.
 That any two instances $\Ids$, $\Ids'$ are equal follows as always in ordinary category theory when one wants to show uniqueness of identity morphisms,
 via
 \begin{equation}
  \Ids \stackrel{{{\lambda'_{\Ids}}^{-1}}}{=} {\Ids' \circ \Ids} \stackrel{\rho_{\Ids'}}{=} \Ids'. 
 \end{equation}
 The rest is an in~\eqref{item:is-prop-1}, as if $\Hom$ is a family of sets, then $\lambda$ and $\rho$ inhabit propositions.

 The final part~\eqref{item:is-prop-3} is a bit trickier. 
 In this case, the additional structure is the type of tuples $(\Ids, \lambda, \rho, t_1, t_0, t_2, h)$, where
 $\Ids$, $\lambda$, $\rho$ are as before, $t_1$ is a witness of the equality~\eqref{eq:id-triangle-1}, $t_0$, $t_2$ are equalities~(\ref{eq:id-triangle-0},~\ref{eq:id-triangle-2}), and $h$ states that $\Hom$ is a family of $1$-types.
 We call $T$ the type of such tuples. 

 Assume we are given a concrete fixed tuple $\mathbf c \defeq (\Ids^c, \lambda^c, \rho^c, t_1^c, t_0^c, t_2^c, h^c) : T$.
 We have to show that under this assumption $T$ is contractible.
 We will show that it is propositional, which suffices.
 The components $t_1$, $t_0$, and $t_2$ are all unproblematic since they inhabit propositions as witnessed by $h^c$, and the same it the case for $h$.
 Thus, we need to show that there is at most one triple $(\Ids, \lambda, \rho)$ such that the properties expressed by $t_1$, $t_0$, $t_2$, $h$ are satisfied.

 Consider the type of pairs $(\Ids, \gamma)$ as in
 \begin{align} 
  & \Ids : \Pi x, \Hom(x,x) \\
  & \gamma: \Pi x, \Ids_x \circ \Ids_x^c = \Ids_x^c.
 \end{align}
 Up to function extensionality and an application of the equivalence $(\_ \circ \Ids_x^c)$, the type of pairs $(\Ids, \gamma)$ is a singleton,
 and hence, there is exactly one such pair $(\Ids, \gamma)$.
 Given this pair, consider triples $(\lambda, \tau, t_0)$ where $\lambda$ and $t_0$ already have the required types and $\tau$ relates $\lambda$ with $\gamma$ as in
 \begin{align}
  & \lambda : \Pi xy, (f : \Hom(x,y)), \Ids_y \circ f = f \\
  & \tau : \Pi x, \lambda_{\Ids^c_x} = \gamma_x \\
  & t_0 :   \Pi xyz, (f : \Hom(x,y)), (g : \Hom(y,z)), \; 
   \lambda_{g \circ f}
   =
   \alpha_{f,g,\Ids} \ct \mapfunc {\_ \circ f} (\lambda_g).
 \end{align}
 We claim that there is at most one triple $(\lambda, \tau, t_0)$.
 If we have such a triple, let us consider the following diagram ($x$ is omitted for readability):
  \begin{equation}
  \begin{tikzpicture}[x=2.2cm,y=-1.5cm, baseline=(current bounding box.center)]
   \node (T1) at (0,0) {$\Ids \circ (\Ids^c \circ f)$};
   \node (T2) at (2,0) {$(\Ids \circ \Ids^c) \circ f$};
   \node (T3) at (1,1) {$\Ids^c \circ f$};
   
   \draw[->] (T1) to node [above] {\scriptsize $\alpha_{f,\Ids^c,\Ids}$} (T2);
   \draw[->] (T1) to node [below left] {\scriptsize $\lambda_{\Ids^c \circ f}$} (T3);
   \draw[->] (T2) to node [above left] {\scriptsize $\mapfunc {\_ \circ f} (\lambda_{\Ids^c})$} (T3);

   \draw[->, bend right=-30] (T2) to node [below right] {\scriptsize $\mapfunc {\_ \circ f} (\gamma)$} (T3);
  \end{tikzpicture}
  \end{equation}
 The triangle commutes thanks to $t_0$ (choosing $g \defeq \Ids^c$), and the two parallel arrows are equal as witnessed by $\tau$.
 Thus, we have
 \begin{equation}
  \lambda_{\Ids^c \circ f} \, = \, \alpha_{f, \Ids^c, \Ids} \ct \mapfunc{\_ \circ f}(\gamma),
 \end{equation}
 and because of $\gamma$ this means that $\lambda_f$ is uniquely determined.
 Due to the restricted truncation level of the family $A_1$, as witnessed by $h^c$, there can be at most one $\tau$ and $t_0$.
 
 Consider the type of pairs $(\rho, t_1)$ of the correct types (the type of $t_1$ is~\ref{eq:id-triangle-1}).
 Again, $t_1$ inhabits a proposition.
 If we look at the diagram~\ref{eq:t1-diagram-picture} which pictures the type of $t_1$, and set $g \defeq \Ids^c$, we see that $\rho$ can be written in terms of all the components that we have so far, and thus is determined as well.
 
 Summarised, we have shown that the type of tuples $(\Ids, \gamma, \lambda, \tau, t_0, \rho, t_1)$ is a proposition.
 By reordering, this type is equivalent to the type of tuples $(\Ids, \lambda, \rho, t_1, t_0, \gamma, \tau)$.
 Observe that the type of pairs $(\gamma, \tau)$ is a singleton and thus contractible no matter what the other components are. 
 Thus, the type of tuples $(\Ids, \lambda, \rho, t_1, t_0)$ is a proposition, and adding the component $t_2$ does not change this since $t_2$ inhabits a proposition as well. 
\end{proof}

Definition~\ref{def:drop-wildness} can be translated directly to the semisimplicial terminology:
\begin{definition}[semi-Segal $n$-type]
 For $n \in \{0,1,2\}$, a semi-Segal $n$-type is an $(n+2)$-restricted semi-Segal type $A$ where $A_1$ is a family of $(n-1)$-types.
\end{definition}
\begin{remark} \label{rem:level-is-trunc->higher-level-is-trunc}
 Given a semi-Segal $n$-type, it is easy to see that $A_i$ is a family of $(n-i)$-truncated types for $i \geq 1$.
 In particular, $A_{n+2}$ is a family of contractible types.
 Nevertheless, we do not want to remove this seemingly trivial level from the definition, as it ensures that the horn-filling/Segal-condition can be formulated uniformly. 
 
 The terminology \emph{semi-Segal $n$-type} comes from the fact that later, after adding completeness, the type $A_0$ will be $n$-truncated.
 At the moment, we cannot draw this conclusion, and $A_0$ could be anything.
\end{remark}

Theorem~\ref{thm:id-structure-unique} implies immediately:
\begin{corollary} \label{cor:truncation-deg-unique}
 Given an $n$-restricted semi-Segal type, where $A_1$ is a family of $(n-3)$-types, there is at most one way to equip it with a degeneracy structure for $n \in \{2,3,4\}$. \qed
\end{corollary}

Before moving on to univalence, we want to mention a simple observation which explains the possible confusing direction reversal of the rightmost vertical arrows in Figure~\ref{fig:low-dimensional-categories}.
Note that, to a semisimplicial type with degeneracies (or wild categorical structure), a truncation condition and a univalence condition can be added independently of each other, both clearly preserving the equivalence between the semisimplicial and the categorical construction.
Thus, we hope the connection of the following statement with the rightmost vertical arrows in Figure~\ref{fig:low-dimensional-categories} is clear, although the figure adds univalence before truncation.

\begin{theorem} \label{thm:on-vertical-arrows-in-fig-1}
 A wild 2-precategory can in a canonical way be seen as a wild precategory, which can be seen as a reflexive-transitive graph.
 For the same structure with an added truncation condition as in Definition~\ref{def:drop-wildness}, this sequence is reversed:
 A preordered set can be seen as a precategory in a canonical way, and analogously, a precategory is a 2-precategory.
\end{theorem}
\begin{proof}
 The first part is of course given by the obvious projections.
 For the second part, note that the components needed to move from a preordered set to a precategory are inhabitants of types that are contractible due to the truncation condition of the preordered set.
 The identical argument works for the step from precategories to 2-precategories.
\end{proof}

Of course, Theorem~\ref{thm:on-vertical-arrows-in-fig-1} translates easily to the formulation using semi-Segal types, where (in the truncated case) the underlying semisimplicial type $(A_1, \ldots, A_{n-1})$ is trivially extended by choosing $A_n$ to be the type family that is constantly $\mathbf 1$.

\subsection{Univalence}

With identity and degeneracy structures at hand, we can 
define what it means for a morphism (or an edge) to be an isomorphism.
We then have two notions of ``sameness'' on objects, namely isomorphism and equality.
Following \citet{ahrens:rezk}, we can assume a univalence principle to collapse these notions into a single one.

To do so, we first need to define what precisely we mean by an \emph{isomorphism}.
However, we need to be careful: there are several ways in which this could be done, and not all of the are well-behaved.
Assume we are given a morphism $f : \Hom(x,y)$ in a wild precategory.
We want a \emph{proposition} $\isIso(f)$.
The obvious definition~\eqref{eq:isIsoAKS} which says $f$ is an isomorphism if there is a morphism $g$ that is both a left and a right inverse falls short of this requirement.
It works for \citet{ahrens:rezk} only because they assume that $\Hom(x,y)$ is always a set.
In our more general setting, an obvious approach is to mirror the definition of a \emph{bi-invertible map}~\cite[Def.~4.2.7 ff.]{hott-book}: 

\begin{definition}[isomorphism in a reflexive-transitive graph] \label{def:iso-in-cat}
 Given a morphism $f : \Hom(x,y)$ in a reflexive-transitive graph or wild (2-) precategory,
 we define the types of \emph{left} and \emph{right inverses} to be
 \begin{equation}
  \mathsf{linv} \defeq \sm{g : \Hom(y,x)} g \circ f = \Ids_x \qquad \text{and} \qquad \mathsf{rinv} \defeq \sm{g : \Hom(y,x)} f \circ g = \Ids_y.
 \end{equation}
 We say that $f$ is an \emph{isomorphism} if it has a left and a right inverse,
 \begin{equation}
  \isIso(f) \defeq \mathsf{linv}(f) \times \mathsf{rinv}(f).
 \end{equation}
\end{definition}

In the setting of semi-Segal types, this definition can be translated as follows:

\noindent
\begin{minipage}[t]{\textwidth -6.2cm}
\begin{lemma}\label{lem:iso-in-sSt}
 An edge $e : A_1(x,y)$ in a 2 (3,4)-restricted semi-Segal type with degeneracies is an \emph{isomorphism} when regarded as a morphism of the corresponding reflexive-transitive graph if and only if the $\Lambda^2_0$ and $\Lambda^2_2$-horns to the right both have contractible filling.
\end{lemma}
\end{minipage}
\hspace{\fill}
\begin{minipage}[t]{6cm}
  \begin{equation} \label{eq:tikz-iso-horn}
  \begin{tikzpicture}[x=1cm*0.7,y=1.73cm*0.7, baseline=(current bounding box.center)]
   \node (T1) at (0,0) {$0$};
   \node (T2) at (2,0) {$1$};
   \node (T3) at (1,1) {$2$};
   
   \draw[->] (T1) to node [below] {$e$} (T2);
   \draw[->] (T1) to node [left] {$s^0_0(x)$} (T3);

   \node (T1B) at (3.5,0) {$0$};
   \node (T2B) at (5.5,0) {$1$};
   \node (T3B) at (4.5,1) {$2$};
   
   \draw[->] (T1B) to node [left] {$s^0_0(y)$} (T3B);
   \draw[->] (T2B) to node [right] {$e$} (T3B);
  \end{tikzpicture}  
  \end{equation}
\end{minipage}

\begin{proof}
This is immediate from Corollary~\ref{cor:composition}.
\end{proof}

While Definition~\ref{def:iso-in-cat} and Lemma~\ref{lem:iso-in-sSt} make sense for a reflexive-transitive graph and 2-restricted semi-Segal type respectively,
they are without further assumptions not very useful in these cases, as not even the identities (or degeneracies) will be isomorphisms.
However, either one more level of structure or a truncation condition will be enough to make the definition of an isomorphism well behaved.

Thus, if $S$ is one of the structures of Definition~\ref{def:id-structure} or Definition~\ref{def:semi-segal-with-deg}, we say that $S$ is \emph{sufficient}
if it either has at least 3 levels of structure (i.e.\ is a wild precategory or a 3-restricted semi-Segal type with degeneracies), or if the morphisms/edges are a family of propositions.

\begin{lemma} \label{lem:sufficient-is-good}
Assume a sufficient structure. Then, for any $x$, the morphism $\Ids_x$ (or edge $s^0_0(x)$) is an isomorphism.
\end{lemma}
\begin{proof}
In a wild precategory, the unitors imply that $\Ids_x$ is its own inverse, hence it is an isomorphism.
On the other hand, if $A_1$ is a family of propositions, then every boundary of a 2-simplex has a contractible filling, hence so do the horns of Definition~\ref{lem:iso-in-sSt}, with $f$ set to $s^0_0(x)$.
\end{proof}

We write
$x \cong y$ for $\Sigma (f : \Hom(x,y)), \isIso(f)$ or for $\Sigma(f : A_1(x,y)), \isIso(f)$, respectively.
We have the familiar function
\begin{equation}
 \mathsf{idtoiso} : x = y \to x \cong y,
\end{equation}
defined by path induction, making use of Lemma~\ref{lem:sufficient-is-good}.

\begin{definition}[univalence]
 A sufficient structure is called \emph{univalent} if the function $\mathsf{idtoiso}$ is an equivalence of types.
\end{definition}

Using the developments of the current and the previous subsection,
we are able to avoid two threats to the well-behavedness of categorical structure.
The first is the issue of wild structure, and the second is the problem that we may have more than one notion of ``sameness''.
Thus, we can record:
\begin{definition}[well-behaved categorical structure] \label{def:well-behaved-uni-structure}
 Combining a truncation condition with univalence, we say:
 \begin{enumerate}[label=$(C{\arabic*})$]
  \item A \emph{poset} (or \emph{univalent $0$-category}) is a reflexive-transitive graph where $\Hom$ is a family of propositions and the univalence condition is satisfied.
  \item A \emph{univalent category} is a univalent wild category 
  where $\Hom$ is a family of sets.
  \item A \emph{univalent 2-category} is a univalent wild 2-category where $\Hom$ is a family of 1-types.
 \end{enumerate}
 The corresponding semisimplicial constructions are the following:
 \begin{enumerate}[label=$(S{\arabic*})$]
  \item A \emph{univalent semi-Segal set} is a 2-restricted semi-Segal type with degeneracies where $A_1$ is a family of propositions, together with the univalence condition.
  \item A \emph{univalent semi-Segal 1-type} is a univalent 3-restricted semi-Segal type with degeneracies such that the family $A_1$ of edges is a family of sets.
  \item A \emph{univalent semi-Segal 2-type} is a univalent 4-restricted semi-Segal type with degeneracies where $A_1$ is a family of 1-types.
 \end{enumerate}
\end{definition}

Regarding the terminology, observe that, in a \emph{univalent semi-Segal $n$-type} (for $n \in \{0,1,2\}$), the type $A_0$ of points is an $n$-type.
This is due to the argument familiar from~\citep{ahrens:rezk} that $x=x$ is equivalent to the $(n-1)$-type of isomorphisms in $A_1$.
We will discuss further related observations in Remark~\ref{rem:flexible-main-definition}.

\begin{theorem} \label{thm:univalent-wellbehaved-equivalent}
 The two lists $(C)$ and $(S)$ in Definition~\ref{def:well-behaved-uni-structure} define equivalent structures.
 Posets correspond to univalent semi-Segal sets,
 univalent categories to univalent semi-Segal 1-types,
 and 
 univalent 2-categories to univalent semi-Segal 2-types.
\end{theorem}
\begin{proof}
 This is a simple extension of the equivalence in Theorem~\ref{thm:identities-degeneracies},
 to which truncation and univalence condition are added on each side.
\end{proof}

We can finally explain the Figure~\ref{fig:low-dimensional-categories} in full.
In the third column (both front and back), we have simply added the condition that the structure in question is univalent.
In the last column, we assume that the structure is truncated.

We also record:

\begin{theorem} [uniqueness of well-behaved categorical structure] \label{thm:uniqueness-of-well-behaved-structure}
 Given an $(n+2)$-restricted semi-Segal type ($n \in \{0,1,2\}$), there is at most one way to equip it with the structure of a univalent semi-Segal $n$-type.
 In other words, the type of structure which turns an $(n+2)$-restricted semi-Segal type into a univalent semi-Segal $n$-type is a proposition.
 The analogous statement holds for reflexive-transitive graphs, for wild semicategories, and for wild 2-semicategories.
\end{theorem}
\begin{proof}
 This follows by combining Theorem~\ref{thm:id-structure-unique} with the fact that the univalence and truncation conditions are propositions.
\end{proof}

\section{Completeness} \label{sec:completeness}

\emph{Completeness} is, after the discussed conditions on horn fillers and truncation levels, the last missing ingredient for a formulation of univalent higher categories which does not require an explicit degeneracy structure.
The results of this section (and of Theorem~\ref{thm:id-structure-unique}) are graphically presented in Figure~\ref{fig:univalence-completeness} on page~\pageref{fig:univalence-completeness}.

\begin{definition}[neutral edges]
\label{def:neutral-edges}
Let $(A_0, \ldots, A_n)$ be an $n$-restricted semi-Segal type ($n \in \{2, 3, 4\}$).  An edge $e : A_1(a, b)$ is said to be \emph{right-neutral} if every outer horn $u : \Lambda^2_0$ with $u_{01} \jdeq e$ has contractible horn filling,
and \emph{left-neutral} if every outer horn $u : \Lambda^2_2$ with $u_{02} \jdeq e$ has contractible horn filling.
Finally, we say that $e$ is \emph{neutral} if it is both right-neutral and left-neutral, and write $\isneut(e)$ for the corresponding proposition.
\end{definition}

\begin{wrapfigure}[6]{r}{5.7cm} 
  \begin{tikzpicture}[x=1cm*0.7,y=1.73cm*0.7, baseline=(current bounding box.center)]
   \node (T1) at (0,0) {$0$};
   \node (T2) at (2,0) {$1$};
   \node (T3) at (1,1) {$2$};
   
   \draw[->] (T1) to node [below] {$e$} (T2);
   \draw[->] (T1) to node [left] {$f$} (T3);

   \node (T1B) at (3.5,0) {$0$};
   \node (T2B) at (5.5,0) {$1$}; 
   \node (T3B) at (4.5,1) {$2$};
   
   \draw[->] (T1B) to node [left] {$g$} (T3B);
   \draw[->] (T2B) to node [right] {$e$} (T3B);
  \end{tikzpicture}  \eqnum 
\end{wrapfigure}
Graphically, we see that $e$ is neutral by definition if the horns to the right have contractible filling, for \emph{any} $f$ which has $a$ as domain and $g$ which has $b$ as codomain.
If we compare the situation to the one represented in~\eqref{eq:tikz-iso-horn}, we see that the only difference is that \emph{neutral} edges have outer horn fillers when combined with any morphism, while for isomorphisms, a degeneracy is required.

If $\C$ is a transitive graph, a \emph{neutral morphism} in $\C$ is a morphism that becomes a neutral edge when $\C$ is regarded as a $2$-restricted semi-Segal type.  As explained above, neutral morphisms will play the role of equivalences.  
This is made precise by the following two lemmata.

\begin{lemma}\label{lem:neutral-equivalence-yoneda}
Let $\C = (\Ob, \Hom, \Comp, \Ids)$ be a transitive graph.  Then a morphism $f : \Hom(x, y)$ in $\C$ is right-neutral (resp. left-neutral) if and only if, for all objects $z : \Ob$, composing with $f$ gives an equivalence $\Hom(y, z) \to \Hom(x, z)$ (resp. $\Hom(w, x) \to \Hom(w, y)$).
\end{lemma}
\begin{proof}
If $h : \Hom(x, z)$, the fibre over $h$ of the map $(\_ \circ f) : \Hom(y, z) \to \Hom(x, z)$ induced by $f$ is precisely given by the type of horn fillers of the horn $u : \Lambda^2_0$ with $u_{01} = f$ and $u_{02} = h$.  It follows that $f$ is right neutral if and only if this fibre is contractible, which is to say that the map $(\_ \circ f)$ is an equivalence.  The analogous statement for left-neutral morphisms is proved similarly.
\end{proof}

\begin{lemma}\label{lem:neutral-equivalence}
Let $\C$ be a wild precategory.  Then for any morphism $f$ of $\C$ the two types $\isneut(f)$ and $\isIso(f)$ are equivalent.  In particular, $\isIso(f)$ is a proposition.
\end{lemma}
\begin{proof}
If $f$ is a neutral edge, then by Lemma~\ref{lem:neutral-equivalence-yoneda} composition with $f$ in both directions gives an equivalence, hence $f$ has a left and a right inverse.  Conversely, if $f$ has a left (resp. right) inverse, then, by associativity of composition in $\C$, it follows that $(\_ \circ f)$ has a right (resp. left) inverse and $(f \circ \_)$ has a left (resp. right) inverse.  In particular, if $f$ has both inverses, then both composition maps are equivalences by Lemma~4.3.3 of~\citep{hott-book}, hence $f$ is neutral by Lemma~\ref{lem:neutral-equivalence-yoneda}.
Therefore, $\isneut(f)$ and $\isIso(f)$ are logically equivalent, so all it remains to show is that $\isIso(f)$ is a proposition.  This now follows immediately from the fact that, for a neutral morphism $f$, the type stating that $f$ is left (resp. right) invertible is equivalent to a singleton, hence it is contractible.
\end{proof}

Although we can fill arbitrary inner horns in a semi-Segal type (cf.~Remark~\ref{rem:on-semiSegal-definition}), nothing is assumed about fillers of \emph{outer horns}.  The definition of neutral edge can then be interepreted as saying that outer horns of the form $\Lambda^2_0$ and $\Lambda^2_2$ can be filled in a semi-Segal type, provided that their \emph{critical edge} is neutral, where the critical edge of an outer horn $u : \Lambda^n_0$ is $u_{01}$, while the critical edge of $v : \Lambda^n_n$ is $v_{n-1, n}$.  More generally, neutral edges in \emph{critical} positions allow us to fill higher dimensional outer horns.
The case that we need is the following:

\begin{lemma}\label{lem:3-special-outer-horns}
Let $A$ be a 3-restricted semi-Segal type, and $x : \Lambda^3_0$ be an outer horn where the critical edge $x_{01}$ is neutral.  Then $x$ has contractible filling.
\end{lemma}
\begin{proof}
This proof follows the same idea as the proof sketched in Remark~\ref{rem:on-semiSegal-definition}.  
If $x : \Lambda^3_0$ is an outer horn where the critical edge is neutral, we want to show that the type $P$ of pairs $(x_{123}, x_{0123})$ is contractible.  Since $x_{01}$ is a neutral edge, the type of pairs $(x_{13}, x_{013})$ is contractible. Therefore, $P$ is equivalent to the type of tuples $(x_{13}, x_{013}, x_{123}, x_{0123})$, which is again contractible, because it can be regarded as a sequence of the two inner horn fillers $(x_{13}, x_{123})$ and $(x_{013}, x_{0123})$.
\end{proof}

With the notion of a neural edge at hand, we can finally formulate \emph{completeness}.
The usefulness of this notion for semi-Segal spaces has been observed by \citet{lurie2009classification} and \citet{harpaz:quasi}.
Our type-theoretic version is the following.
\begin{definition}[completeness]
\label{def:restr-csst}
An $n$-restricted ($n \in \{2,3,4\}$) semi-Segal type $A \jdeq (A_0, \ldots, A_n)$ is \emph{complete} if,
for every point $x$, there is a unique neutral morphism with codomain $x$:
\begin{equation}
\Pi(x:A_0), \iscontr\left(\Sigma(y:A_0), (e : A_1(y,x)), \isneut(e)\right).
\end{equation}
\end{definition}

As said earlier, completeness is equivalent to univalence whenever the latter makes sense.

\begin{lemma} \label{lem:uni-equals-complete}
Let $A$ be a 3-restricted (or 4-restricted) semi-Segal type with degeneracies.  Then $A$ is complete if and only if it is univalent as a wild precategory.
\end{lemma}
\begin{proof}
Fix a point $x : A_0$.
For any point $y : A_0$, we have a sequence of functions
\begin{equation}
\begin{tikzpicture}[x=5cm, baseline=(current bounding box.center)]
  \node (A) {$(x = y)$};
  \node[right of=A, node distance=4cm] (B) {$\Sigma(e : A_1(x,y)), \isIso(e)$};
  \node[right of=B, node distance=5.5cm] (C) {$\Sigma(e : A_1(x,y)), \isneut(e)$};
  \draw[->, shorten <=0.2cm, shorten >=0.2cm] (A) to node [above] {\scriptsize $\mathsf{idtoiso}$} (B);
  \draw[->, shorten <=0.2cm, shorten >=0.1cm] (B) to node [above] {\scriptsize $\mathsf{isotoneut}$} (C);
\end{tikzpicture}
\end{equation}
where the second function is defined and an equivalence by Lemma~\ref{lem:neutral-equivalence}.
If we pass to the total spaces as in~\cite[Definition 4.7.5]{hott-book}, we get a function
\begin{equation} \label{eq:total-id-to-neutral}
 \mathsf{total}(\mathsf{isotoneut} \circ \mathsf{idtoiso}) : \, \Sigma(y:A_0), x=y \, \longrightarrow \, \Sigma(y:A_0), (e : A_1(x,y)), \isneut(e).
\end{equation}
Since its domain is a singleton,~\eqref{eq:total-id-to-neutral} is an equivalence for all $x$ if and only if $\Sigma(y:A_0), (e : A_1(x,y)), \isneut(e)$ is always contractible, i.e.\ if and only if $A$ is complete.
At the same time,~\eqref{eq:total-id-to-neutral} is an equivalence if and only if $\mathsf{isotoneut} \circ \mathsf{idtoiso}$ is an equivalence by~\cite[Theorem 4.7.7]{hott-book}, thus by 2-out-of-3 exactly if $\mathsf{idtoiso}$ is an equivalence for all $x,y$.
\end{proof}

Perhaps surprisingly, one can always construct a degeneracy structure for a complete semi-Segal type.

\begin{lemma}\label{lem:1-completeness-degeneracies}
Let $A$ be a $2$-restricted complete semi-Segal type.  Then $A$ admits a $1$-degeneracy structure.
\end{lemma}
\begin{proof}
We need to show that for every point $y : A_0$, there exists an edge $s^0_0(y) : A_1(y, y)$.  
By completeness, we can find a point $x : A_0$ and a neutral edge $e : A_1(x, y)$.  
Now consider the horn $u : \Lambda^2_0$, where $u_{01} \defeq u_{02} \defeq e$.  
Since $e$ is neutral, we can fill $u$ to a full triangle which we denote by $S^0_0(y)$.
Observe that the face $(S^0_0(y))_{12}$ gives us an edge in $A_1(u_1, u_2) \jdeq A_1(y, y)$, and we define $s^0_0(y)$ to be this edge.
\end{proof}

\begin{lemma}\label{lem:2-completeness-degeneracies}
Let $A$ be a $3$-restricted complete semi-Segal type. Then $A$ admits a $2$-degeneracy structure.
\end{lemma}
\begin{proof}
We need to construct $s^1_0$ and $s^1_1$.
Let $f : A_1(x, y)$ and $g : A_1(y,z)$ be two edges in $A$ (a single edge would be enough here, we only use two different ones to facilitate the proof of Lemma~\ref{lem:3-completeness-degeneracies} below).
In both figures, the triangle composed of $e$'s and $s^0_0(y)$ is the triangle $S^0_0(y)$ constructed in the previous lemma.

\noindent
\begin{minipage}{\textwidth -6.25cm}
In the left diagram, let $p : \fulltetra 2$ be the triangle we get by filling the $\Lambda^2_1$-horn given by $e$, $g$ (i.e.\ we have $p_{01} \jdeq e$ and $p_{12} \jdeq g$).
We can define a $\Lambda^3_0$-horn $u$ by setting $u_{013} \defeq u_{023} \defeq p$, and $u_{012} \defeq S^0_0(y)$.
By Lemma~\ref{lem:3-special-outer-horns} we can fill $u$, giving us a full tetrahedron $S^1_0(g) : \fulltetra 3$, and we set $s^1_0(g)$ to be the triangle filler provided by $(S^1_0(g))_{123}$. 
\end{minipage}%
\hspace{\fill}
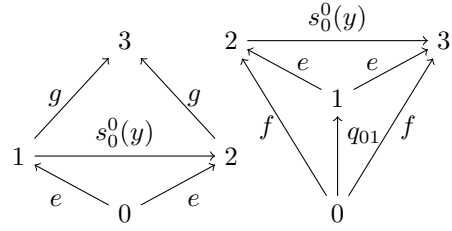
\begin{minipage}{6cm}
\begin{equation*}
\begin{tikzpicture}[x=0.7cm,y=1.53cm, baseline=(current bounding box.center)]
  \node (T3) at (2,1) {$3$};
  \node (T1) at (0,0) {$1$};
  \node (T2) at (4,0) {$2$};
  \node (T0) at (2,-0.5) {$0$};

  \draw[->] (T0) to node [below left] {$e$} (T1);
  \draw[->] (T0) to node [below right] {$e$} (T2);
  \draw[->] (T1) to node [left] {$g$} (T3);
  \draw[->] (T2) to node [right] {$g$} (T3);
  \draw[->] (T1) to node [above] {$s^0_0(y)$} (T2);
  
  \node (U0) at (6,-0.5) {$0$};
  \node (U1) at (6,0.5) {$1$};
  \node (U2) at (4,1) {$2$};
  \node (U3) at (8,1) {$3$};
  
  \draw[->] (U0) to node [left] {$f$} (U2);
  \draw[->] (U0) to node [right] {$f$} (U3);
  \draw[->] (U1) to node [above, near start] {$e$} (U2);
  \draw[->] (U1) to node [above, near start] {$e$} (U3);
  \draw[->] (U2) to node [above] {$s^0_0(y)$} (U3);
  \draw[->] (U0) to node [right, near end] {$q_{01}$} (U1);
\end{tikzpicture}
\end{equation*}
\end{minipage}%

In the right diagram, we first fill the $\Lambda^2_2$-horn determined by $e$ and $f$, giving us $q : \fulltetra{2}$ with $q_{02} \jdeq f$ and $q_{12} \jdeq e$, and $q_{01}$ as shown in the diagram. 
We define a $\Lambda^3_1$-horn $v$ by $v_{012} \defeq v_{013} \defeq q$ and by choosing $v_{123} \defeq S^0_0(y)$.
The filler for the horn $v$ gives us a full tetrahedron $S^1_1(f) : \fulltetra 3$, 
and we define $s^1_1(f)$ to be the filler given by the face $(S^1_1(f))_{023}$.
\end{proof}

\begin{lemma}\label{lem:3-completeness-degeneracies}
Let $A$ be a 4-restricted complete semi-Segal type. Then $A$ admits a $3$-degeneracy structure.
\end{lemma}
\begin{proof}
We proceed similarly as in the previous lemmata.
By Corollary~\ref{cor:deg-derivable}, we only need to construct $s^2_1$.  Let $\alpha : \fulltetra 2$ be any triangle composed of $f : A_1(x,y)$, $g : A_1(y,z)$ and $h : (x,z)$, together with a filler.

\noindent
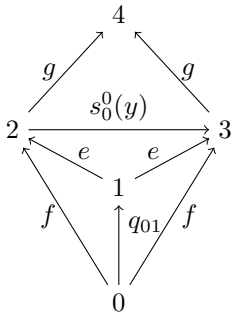
\begin{minipage}{3.7cm}
\begin{equation*}
\begin{tikzpicture}[x=0.7cm,y=1.53cm, baseline=(current bounding box.center)]
  \node (U0) at (6,-0.5) {$0$};
  \node (U1) at (6,0.5) {$1$};
  \node (U2) at (4,1) {$2$};
  \node (U3) at (8,1) {$3$};
  \node (U4) at (6,2) {$4$};
  
  \draw[->] (U0) to node [left] {$f$} (U2);
  \draw[->] (U0) to node [right] {$f$} (U3);
  \draw[->] (U1) to node [above, near start] {$e$} (U2);
  \draw[->] (U1) to node [above, near start] {$e$} (U3);
  \draw[->] (U2) to node [above] {$s^0_0(y)$} (U3);
  \draw[->] (U0) to node [right, near end] {$q_{01}$} (U1);
  \draw[->] (U2) to node [left] {$g$} (U4);
  \draw[->] (U3) to node [right] {$g$} (U4);
\end{tikzpicture}
\end{equation*}
\end{minipage}%
\hspace{\fill}
\begin{minipage}{\textwidth-3.8cm}
The diagram on the left, where the edge $h$ is not drawn, can be viewed as putting the tetrahedra $S^1_1(f)$ and $S^1_0(g)$ from the previous lemma on top of each other.

We now construct a $\Lambda^4_1$-horn $u$ as follows:
\begin{itemize}
\item $u_{0123} \defeq S^1_0(f)$;
\item $u_{1234} \defeq S^1_1(g)$;
\item $u_{124} \jdeq u_{134}$ is the filler of the $\Lambda^2_1$-horn determined by $e$ and $g$;
\item $u_{0124} \jdeq u_{0134}$ is the filler of the $\Lambda^3_2$-horn determined by $u_{012}$, $\alpha$, and $u_{124}$.
This horn is fillable by the argument of Remark~\ref{rem:on-semiSegal-definition}.
\end{itemize}
The filler of $u$ then contains the required degeneracy as the face $u_{0234}$. \qedhere
\end{minipage}
\end{proof}

\begin{corollary}
 Any wild complete $n$-semicategory ($n \in \{0,1,2\}$) can be equipped with an identity structure.  \qed
\end{corollary}

\section{Conclusions} \label{sec:conclusions}

We have given definitions of categorical structures in HoTT based on semi-simplicial types and proved that they are equivalent to existing notions of univalent categories~\citep{ahrens:rezk}.  
Putting pieces together, the main new notion that we consider is that of a \emph{complete semi-Segal $n$-type}, which we formally only have defined for $n \in \{0,1,2\}$. Unfolded, it can be stated as:

\begin{definition}[complete semi-Segal $n$-type] \label{def:final-definition}
 A \emph{complete semi-Segal $n$-type} is semisimplicial type $(A_0, \ldots, A_{n+2})$, equipped with three properties:
 First, all $\Lambda^k_1$-horns ($k \geq 2$) have contractible filling. Second, \emph{completeness} is satisfied. Third, $A_1$ is a family of $(n-1)$-types.
\end{definition}

The remarkable feature of this definition is that each of the three properties is a proposition.
We present our main result as follows:

\begin{theorem}\label{thm:main-result}
 The type of \emph{complete semi-Segal $n$-types} is equivalent to the type of \emph{univalent $n$-categories}, for $n \in \{0,1,2\}$.
\end{theorem}
\begin{proof}
 Theorem~\ref{thm:univalent-wellbehaved-equivalent} says that univalent $n$-categories are the same as univalent semi-Segal $n$-types.
 Univalent semi-Segal $n$-types are, by definition, $(n+2)$-restricted semi-Segal types with degeneracies, with a univalence condition, and the condition that $A_1$ is a family of $(n-1)$-types.
 By Lemma~\ref{lem:uni-equals-complete}, we can substitute the univalence condition by a completeness condition, without changing the type up to equivalence.
 Corollary~\ref{cor:truncation-deg-unique} guarantees that the type corresponding to the degeneracy structure is a proposition.
 By Lemmata~\ref{lem:1-completeness-degeneracies},~\ref{lem:2-completeness-degeneracies}, and~\ref{lem:3-completeness-degeneracies}, this type is inhabited, therefore contractible, hence we can remove the condition.
 What remains is exactly the type of complete semi-Segal $n$-types.
\end{proof}

Thanks to the theorem, we propose to employ this notion of complete semi-Segal $n$-type as the definition of a \emph{univalent $n$-category}.
Of course, it is not yet clear whether this particular model of $(n,1)$-category will turn out to be practically useful in the development of HoTT.
One problem with our approach is that, due to the well-known limitations about representing semi-simplicial types or any form of infinite tower of coherences internally in HoTT, we cannot state Definition~\ref{def:final-definition} for a variable $n : \N$.
If however we are confronted with a concrete problem for which we want to use higher categories, and we know a fixed $n_0$ such that univalent $n_0$-categories are sufficient, then we can express the definition internally in HoTT and, for example, formalise the argument in Agda~\citep{norell:towards}, Coq~\citep{coq}, or Lean~\citep{moura:lean}.

To make full use of this notion of univalent $n$-categories, their theory needs to be developed.
If one is happy to work in a stronger system extending HoTT such as HTS~\citep{voe:hts} or two-level type theory~\citep{alt-cap-kra:two-level}, the general case (including univalent $\infty$-categories) can be formalised and studied.
We have started to translate some results by \citet{lurie:higher-topoi} into our setting and at least in parts this works very nicely, but it is at the time of writing too early to further report on this.
A very preliminary demonstration is given by \citet{ann-cap-kra:two-level},
where the basic results about Reedy fibrant semisimplicial types were formalised by embedding two-level type theory in the proof assistant Lean.
Thanks to the conservativity results by \citet{paolo:thesis}, such results can at least for a fixed $n_0 < \infty$ be ``transported back'' to pure HoTT, although the details of this translation are still subject to ongoing research.
For all concrete constructions which we have performed, we found it easy to do this translation by hand.
For example, \citet{ann-cap-kra:two-level} show that the type universe is an $(\infty,1)$-category, and it is clear that the given construction can be used to 
see how the universe restricted to $(n-1)$-types is a univalent $n$-category; it can probably be regarded as the prototypical example.

Independently of this, a natural question seems to be whether the definition of $n$-categories using explicit composition and identity structure (as in Section~\ref{sec:composition-and-horns} and~\ref{sec:identity-and-degeneracy}) can be done for $n > 2$, and whether Theorem~\ref{thm:main-result} can be extended to this case.
For concrete and very low $n$ ($3$ and possibly $4$), one should with enough patience be able to write down the appropriate definitions and work out whether they are equivalent.
For the more general case however, the combinatorial aspects of higher associahedra require much more sophistication, and, to the best of the authors' knowledge, have not yet been worked out in a context that is general enough to be applicable to HoTT.
If one managed to find a representation of $n$-categories with explicit composition and identities (for any externally fixed $n$),
it seems plausible that a version of Theorem~\ref{thm:main-result} could be shown.
However this is in no way guaranteed,
since phenomena occurring at higher levels, like the fact that a tricategory is in general not equivalent to a strict $3$-category,
do not show up in the cases $n \leq 2$ which we have dealt with.

What we know is that we can always construct a degeneracy structure for a given complete semi-Segal $n$-type $A$.
We do not show this in the current paper (where the whole technical development is restricted to the case $n \leq 2$), but it follows with the help of the argument sketched by \citet{kraus-sattler:spacediagrams} 
(Theorem 5.1(2), where the index category $I$ is instantiated with a finite total order and $\mathsf{T}$ is replaced by $A$) 
and shows that Lemmata~\ref{lem:1-completeness-degeneracies}-\ref{lem:3-completeness-degeneracies} can be done for general $n$.
As soon as one has a degeneracy structure, it is easy to see that there is some flexibility in the formulation of Definition~\ref{def:final-definition}.
For emphasis, we formulate this as a remark:

\begin{remark}[equivalent definitions of complete semi-Segal $n$-types] \label{rem:flexible-main-definition}
 If $(A_0, \ldots, A_{n+2})$ is a semi-Segal type satisfying the completeness condition, and $A_{i+1}$ is a family of $k$-types, then $A_{i}$ is a family of $(k+1)$-types.
 As Remark~\ref{rem:level-is-trunc->higher-level-is-trunc} suggests, the reverse is true as long as $i \geq 1$.
 Thus, the truncation condition of Definition~\ref{def:final-definition} could equivalently be formulated by fixing any $i$ with $1 \leq i \leq n+2$ and saying that $A_i$ is a family of $(n-i)$-types.
 We have done this using $i \defeq 1$.
 The other canonical choice would have been $i \defeq n+2$.
 Note that stating the truncation condition only for $i \jdeq 0$ is insufficient, which seems to be a weakness of the suggestion by \citet{schreiber_nlab_inf1cat}.
 Alternative ways of phrasing the Segal condition have been discussed in Remark~\ref{rem:on-semiSegal-definition}; the version proposed by \citet{schreiber_nlab_inf1cat} is $\fulltetra{p} \to \spine{p}$ being an equivalence.
\end{remark}

Another interesting question is whether Corollary~\ref{cor:truncation-deg-unique} holds for $n > 2$, i.e.\ whether an appropriately truncated restricted (but not necessarily complete) semi-Segal type has at most one degeneracy structure.
We have not worked out a proof for this.
It is not even clear to us whether a \emph{complete} semi-Segal $n$-type necessarily has a contractible degeneracy structure for $n > 2$.
Although the mentioned argument of \citet{kraus-sattler:spacediagrams} does give us more than a degeneracy structure,
namely the contractibility of a certain type, this type is something like a ``double degeneracy structure'' rather than a degeneracy structure.

Moreover, note that our notion of degeneracy structure does not take the simplicial identity~\eqref{eq:simpli-id-s-s}, i.e.\ $s^k_i \circ s^{k+1}_j = s^k_{j+1}$, into account.
For the considered case of a complete semi-Segal $2$-type, Christian Sattler has pointed out to us that the first instance of this equation would be an equation in $A_2$ and thus a proposition.
Therefore, it is possible that the absence of~\eqref{eq:simpli-id-s-s} remained unnoticed for the cases we have considered but would play a role on later levels.
However, we do not think that this consideration is a problem for Definition~\ref{def:final-definition}, which we believe gives a well-behaved structure without requiring us to decide which notion of degeneracies is ``correct''.

To continue with the discussion of the omitted simplicial identity~\eqref{eq:simpli-id-s-s},
note that the very first instance of it would, when translated to the terminology of wild 2-precategories, give an equation $\lambda_{\Ids} = \rho_{\Ids}$. 
If we look at Remark~\ref{rem:two-triangles-too-much}, one might argue that such an equation does live at the ``same level'' as $t_0$ and $t_2$, see (\ref{eq:id-triangle-0},\ref{eq:id-triangle-2}), and should therefore have been included in Definition~\ref{def:degeneracy-structure}.
We think the difference is that, unlike (\ref{eq:id-triangle-0},\ref{eq:id-triangle-2}), the coherence $\lambda_{\Ids} = \rho_{\Ids}$ never has to be mentioned in the definition of higher categories (e.g.\ it is absent in the set-based definition of tricategories by \citet{nick:tricats})
since the one derivable from the triangle coherences is automatically coherent in some sense.

Even the restriction to $2$-categories that we have discussed in this paper already allows the formulation of many interesting examples of categorical structures in HoTT that were previously not obtainable, such as the category of univalent categories (not capturing all natural transformations), and (the ``homwise'' core of) the bicategory of spans of a finitely complete univalent category.
For $1$-categories, we note that completeness enables a slick representation of univalent categories, since we only need objects, sets of morphisms, composition, associativity, and completeness.
It is possible that this definition is convenient for the development of standard category theory in HoTT, but we have not investigated the idea.

The restriction to low dimensions also allowed us to produce a formalisation in the proof assistant Agda which is reasonably close to the informal text.  Our formalisation 
is based on the HoTT core library \texttt{agda-base} by the first named author, and 
covers all the equivalences presented in Figure~\ref{fig:low-dimensional-categories}.
This means that we have in particular formalised the equivalences of $n$-restricted semi-Segal types and wild $n$-semicategories ($n \in \{0,1,2\}$) in full, as well as the respective equivalence of degeneracy and identity structures.
The equivalences have been carefully constructed so that they ``compute'' in the expected way.  For example, the equality~\eqref{eq:jdgm-A2-comp} of Corollary~\ref{cor:composition} holds judgmentally.

Higher categories are clearly related to \emph{directed type theory}~\citep{LICATA2011263,nuyts:masterthesis}, 
where one considers theories that have types corresponding to $\infty$-categories rather than $\infty$-groupoids.
Recent work by \citet{rhiel-shulman:directed} 
considers categories (externally) that are equipped with a ``directed interval'', 
such as bisimplicial sets, and uses an enriched version of the language of HoTT 
to give a definition of \emph{Segal} and \emph{Rezk types}.
In comparison, we use ``standard'' HoTT and look at semi-Segal objects there, in the conventional sense.
One possible way to relate the two approaches would be to say that our
construction can be regarded as a way to give a semantics to the
theory by \citet{rhiel-shulman:directed} based on a model of ``standard'' HoTT, 
although this is currently a vague statement and a significant amount of work would be required to make it precise.

\subsection*{Acknowledgments}
We would like to thank Christian Sattler for countless insightful discussions on topics related to this work, and for helpful comments on an earlier draft.
We are also grateful to Thorsten Altenkirch for many inspiring conversations on type-theretic $\infty$-categories,
as well as to Andreas Nuyts for a discussion which motivated us to write this paper.
Further, we thank Gershom Bazerman and Ulrik Buchholtz for their comments on an earlier draft, as well as the anonymous referees for their throughout reviews that have helped us to improve clarity and motivation.
Finally, we thank Lars Birkedal who checked the final version of this article.

\bibliographystyle{plainnat}
\bibliography{master}

\end{document}